\documentclass[STIX1COL]{WileyNJD-v2}

\usepackage{amsmath}
\usepackage{amsfonts}
\usepackage{amssymb}
\usepackage{amsopn}
\usepackage{graphicx}
\usepackage{xspace}
\usepackage{bm}
\usepackage{threeparttable}
\usepackage{subfigure}
\usepackage{afterpage}
\usepackage{soul}

\articletype{Article Type}%

\received{<day> <Month>, <year>}
\revised{<day> <Month>, <year>}
\accepted{<day> <Month>, <year>}

\newcommand{\sss}[0]{\scriptscriptstyle}

\newcommand{\mat}[1]{\ensuremath{\mathsf{#1}}}

\newcommand{\ignore}[1]{}
\newcommand{\eg}[0]{{e.g.\@}\xspace}
\newcommand{\ie}[0]{{i.e.\@}\xspace}

\newcommand{\ubc}[0]{u_{\sss\Gamma}}
\newcommand{\ubcn}[0]{u_{{\sss\Gamma},n}}
\newcommand{\tOmega}[0]{\tilde{\Omega}}
\newcommand{\tGamma}[0]{\tilde{\Gamma}}
\newcommand{\tlambda}[0]{\tilde{\lambda}}

\newcommand{\Htrc}[0]{H^{\sss\frac{1}{2}}}

\renewcommand{\a}[0]{\ensuremath{\mathfrak{a}}}
\renewcommand{\b}[0]{\ensuremath{\mathfrak{b}}}
\newcommand{\bh}[0]{\ensuremath{\mathfrak{b}_h}}
\renewcommand{\l}[0]{\ensuremath{\mathfrak{l}}}
\newcommand{\aform}[2]{\ensuremath{\a\!\left(#1,#2\right)}}
\newcommand{\bform}[2]{\ensuremath{\b\!\left(#1,#2\right)}}
\newcommand{\bhform}[2]{\ensuremath{\bh\!\left(#1,#2\right)}}
\newcommand{\lform}[1]{\ensuremath{\l\!\left(#1\right)}}
\newcommand{\dH}[2]{\ensuremath{d_{H}(#1,#2)}}
\newcommand{\Jreg}[0]{\ensuremath{J_{\text{reg}}}}
\newcommand{\Vh}[0]{\ensuremath{V_{h}}}
\newcommand{\Vhp}[0]{\ensuremath{V_{h,p}}}
\newcommand{\uh}[0]{\ensuremath{\bar{u}_{h}}}
\newcommand{\vh}[0]{\ensuremath{\bar{v}_{h}}}
\newcommand{\ch}[0]{\ensuremath{\bar{c}_{h}}}
\newcommand{\fh}[0]{\ensuremath{\bar{f}_{h}}}
\newcommand{\bvech}[0]{\ensuremath{\bar{b}_{h}}}
\newcommand{\psih}[0]{\ensuremath{\bar{\psi}_{h}}}

\newcommand{\Jump}[1]{[\!\![ #1 ]\!\!]} 
\newcommand{\Mean}[1]{\{\!\!\{#1\}\!\!\}}

\newcommand{\diff}[0]{\mathrm{d}}

\newcommand{\Huu}[0]{\mat{H}_{uu}}

\newcommand{\Hcu}[0]{\mat{H}_{cu}}
\newcommand{\Hcc}[0]{\mat{H}_{cc}}
\newcommand{\Hz}[0]{\mat{H}}
\newcommand{\Au}[0]{\mat{A}_{u}}
\newcommand{\Ac}[0]{\mat{A}_{c}}

\newcommand{\RevOneAdd}[1]{#1}
\newcommand{\RevOneDel}[1]{} 
\definecolor{green}{RGB}{51,153,51}
\newcommand{\RevTwoAdd}[1]{#1}
\newcommand{\RevTwoDel}[1]{} 

\begin{document}

\title{An inverse problem formulation of the immersed boundary method}

\author[1]{Jianfeng Yan}
\author[1]{Jason E. Hicken*}
  
\authormark{YAN and HICKEN}

\address[1]{\orgdiv{Department of Mechanical, Aerospace, and Nuclear Engineering}, \orgname{Rensselaer Polytechnic Institute}, \orgaddress{\state{New York}, \country{United States of America}}}

\corres{*Jason E. Hicken, Department of Mechanical, Aerospace, and Nuclear Engineering,
  Rensselaer Polytechnic Institute, Troy, New York, 12180, United States. \email{hickej2@rpi.edu}}

\abstract[Abstract]{We formulate the immersed-boundary method (IBM) as an inverse problem.  A control variable is introduced on the boundary of a larger domain that encompasses the target domain.  The optimal control is the one that minimizes the mismatch between the state and the desired boundary value along the immersed target-domain boundary.  We begin by investigating a na\"ive problem formulation that we show is ill-posed: in the case of the Laplace equation, we prove that the solution is unique but it fails to depend continuously on the data; for the linear advection equation, even solution uniqueness fails to hold.  \RevTwoAdd{These issues are addressed} by two complimentary strategies.  The first strategy is to ensure that the enclosing domain tends to the true domain as the mesh is refined.  The second strategy is to include a specialized parameter-free regularization that is based on penalizing the difference between the control and the state on the boundary.  The proposed inverse IBM is applied to the diffusion, advection, and advection-diffusion equations using a high-order discontinuous Galerkin discretization.  The numerical experiments demonstrate that the regularized scheme achieves optimal rates of convergence and that the reduced Hessian of the optimization problem has a bounded condition number as the mesh is refined.}
    

\maketitle



\section{Introduction}

Mesh generation remains a significant bottleneck for many aerospace applications
of computational fluid dynamics (CFD)~\cite{cfd2030} and the numerical solution
of partial differential equations (PDEs) more generally.  This bottleneck is
particularly acute during the generation and adaptation of curved-element,
anisotropic meshes around complex geometries, which has limited the adoption of
high-order methods in industry.

One way to address the meshing bottleneck is to develop a high-order
discretization that does not require a conforming mesh.  The immersed boundary
method (IBM) offers a potential framework for constructing such a discretization
and forms the basis for the method presented herein.

Before proceeding, we should distinguish the IBM from immersed-interface, or
cut-cell, methods.  IBMs~\cite{peskin:1977} impose the boundary conditions
indirectly through a body force or modified boundary flux.  Immersed-interface
methods, on the other hand, modify the cells, elements, or stencil near the
boundary such that the boundary condition can be applied
directly~\cite{Purvis1979prediction, Berger1989adaptive, LeVeque1994immersed,
  aftosmis:1998, Hansbo2002unfitted, fidkowski:2007, Lew2008discontinuous,
  Brehm2013novel, Huafei2013thesis, Mattsson2017embedded}.  \RevOneDel{Immersed-interface
methods offer an alternate approach to addressing the meshing bottleneck, but
they are not without their difficulties.  The process of ``cutting'' cells and
elements is a non-trivial task whose complexity is arguably on par with
geometry-conforming mesh generation.  Furthermore, the cutting process
invariably creates relatively small and non-standard element shapes or stencils
that must be treated carefully to avoid poor conditioning and/or nonlinear
solver convergence issues~[11] 
These problems are compounded when considering high Reynolds number flows that necessitate highly
stretched grids to accurately resolve the boundary layer.}  The present work is more closely related to classical IBMs that introduce a body
force, or penalty, that imposes the boundary conditions indirectly.  The IBM
framework encompasses a wide range of approaches, so a comprehensive review is
beyond the scope of this work.  For reviews of the IBM in the context of
finite-difference/volume methods see~\cite{Peskin2002immersed} and Sections 3
and 4.1 of \cite{mittal:2005}.  A review of penalty-based IBMs used in
finite-element methods can be found in Section 2.2 of
\cite{Lew2008discontinuous}.

IBMs are popular in some CFD applications, but they have not found
widespread use for the steady, advection-dominated problems common in the aerospace industry.  We believe this is primarily due to their limited accuracy; most IBMs are first-order accurate in practice.  That said, some high-order IBMs have been proposed and we mention a few here.

Mayo~\cite{Mayo1984fast} solved the Poisson's and biharmonic equations on
irregular domains by constructing a discontinuous extension of the solution onto
a regular (\eg square) domain.  The discontinuities in the solution and its
derivatives could then be determined by solving an integral equation, and
these jumps could subsequently be introduced into the Poisson's or biharmonic
equations to solve the PDEs on the regular domain using fast solvers.  Mayo
later extended this technique to 4th order in \cite{Mayo1992rapid}.

Marques~\etal~\cite{Marques2011correction} and
Marques~\cite{Marques2012correction} generalized the ghost-fluid
method~\cite{Fedkiw1999ghost, Fedkiw1999nonoscillatory, Liu2000boundary,
  Kang2000boundary} to high-order by defining a correction function that
smoothly extends the solution on either side of an interface or boundary.  Once
the correction function is determined, it is included on the right-hand side of
the discretized equations.  The correction is independent of the solution for
linear problems with Dirichlet boundary conditions, so conditioning of the
left-hand side is not affected in this case; however, more generally, the
correction function depends on the solution for nonlinear problems and
problems with more general boundary conditions.

\RevTwoAdd{Finally, IBMs that impose boundary conditions using Nitsche's method have shown considerable promise.  The penalty parameter in Nitsche-based methods must be chosen carefully to maintain coercivity, and this choice can be challenging since it depends on how the exact boundary intersects with the computational mesh.  However, this challenge can be alleviated with stabilization introduced by penalizing the jumps in the normal derivatives~\cite{cutfem2015}.  Although high-order Nitsche-based IBMs have been applied to several different types of physics, their application to advection-dominated problems is lacking, to the best of our knowledge.}

In this paper we present the preliminary investigation of a novel, high-order
immersed boundary method.  Our approach is similar in spirit to the fictitious domain method of Glowinski and He~\cite{glowinski_he_2011} and the optimization-based interface treatment in~\cite{Kuberry2017optimization}.  The
key idea is to formulate the IBM as an inverse problem in which unknown boundary
fluxes, defined on an approximate easy-to-generate boundary, are used to satisfy
the desired boundary conditions on the true boundary.  One strength of this
high-order inverse-problem formulation of the IBM is that it is straightforward both
conceptually and in its implementation.

The remainder of the paper is organized as follows.  Section~\ref{sec:formulation} describes the basic inverse IBM formulation for both the Laplace and linear advection equations.  This section also discusses the ill-posedness of the basic formulation.  Section~\ref{sec:cond} presents an investigation of a model problem, which we use to better understand the nature of ill-conditioning in the discrete setting. Regularization of the inverse IBM is discussed in Section~\ref{sec:regularize}, including a parameter-free regularization based on penalizing the control against the state.  The proposed method is demonstrated in Section~\ref{sec:results} using on a discontinuous Galerkin discretization.  A summary and discussion are provided in Section~\ref{sec:conclude}.

\section{Inverse IBM problem formulation and ill-posedness}\label{sec:formulation}

This section presents the basic formulation of the proposed IBM in the context of the Laplace and linear advection equations.  In both cases, we show that the basic formulation is ill-posed and requires regularization.

\subsection{Application to the Laplace equation}

Consider the Laplace equation on the open, bounded domain $\Omega \in \mathbb{R}^{N}$, with Dirichlet boundary conditions applied on the boundary $\Gamma \equiv \partial \Omega$:
\begin{subequations}\label{eq:laplace}
\begin{alignat}{2}
  \nabla^2 u &= 0,
  &\qquad &\forall x \in \Omega, \label{eq:laplace_pde} \\
  u &= \ubc,&\qquad &\forall x \in \Gamma. \label{eq:laplace_bcs}
\end{alignat}
\end{subequations}
Suppose $\Omega$ is a geometrically complex domain for which we do not want to generate a conforming mesh.  Consequently, we introduce a geometrically simpler, bounded domain $\tOmega \supseteq \Omega$, with boundary $\tGamma \equiv \partial \tOmega$.  Figure~\ref{fig:laplace_geometry} depicts an example of the domain $\Omega$, an encompassing domain $\tOmega$, and their respective boundaries.
\begin{figure}[t]
  \begin{center}
    \subfigure[ \label{fig:laplace}]{%
      \includegraphics[width=0.45\textwidth]{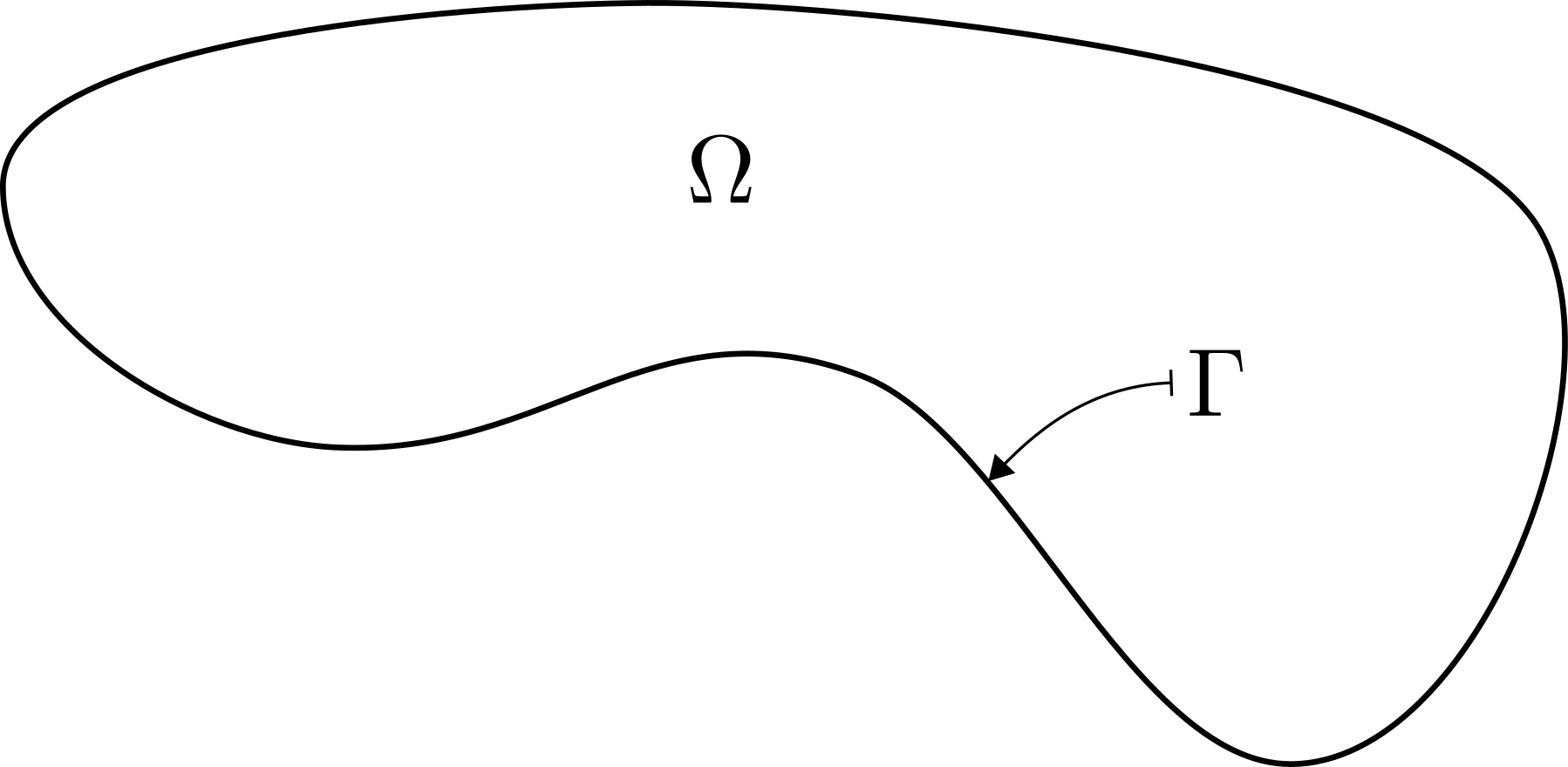}}\hfill
    \subfigure[ \label{fig:laplace_III}]{%
      \includegraphics[width=0.45\textwidth]{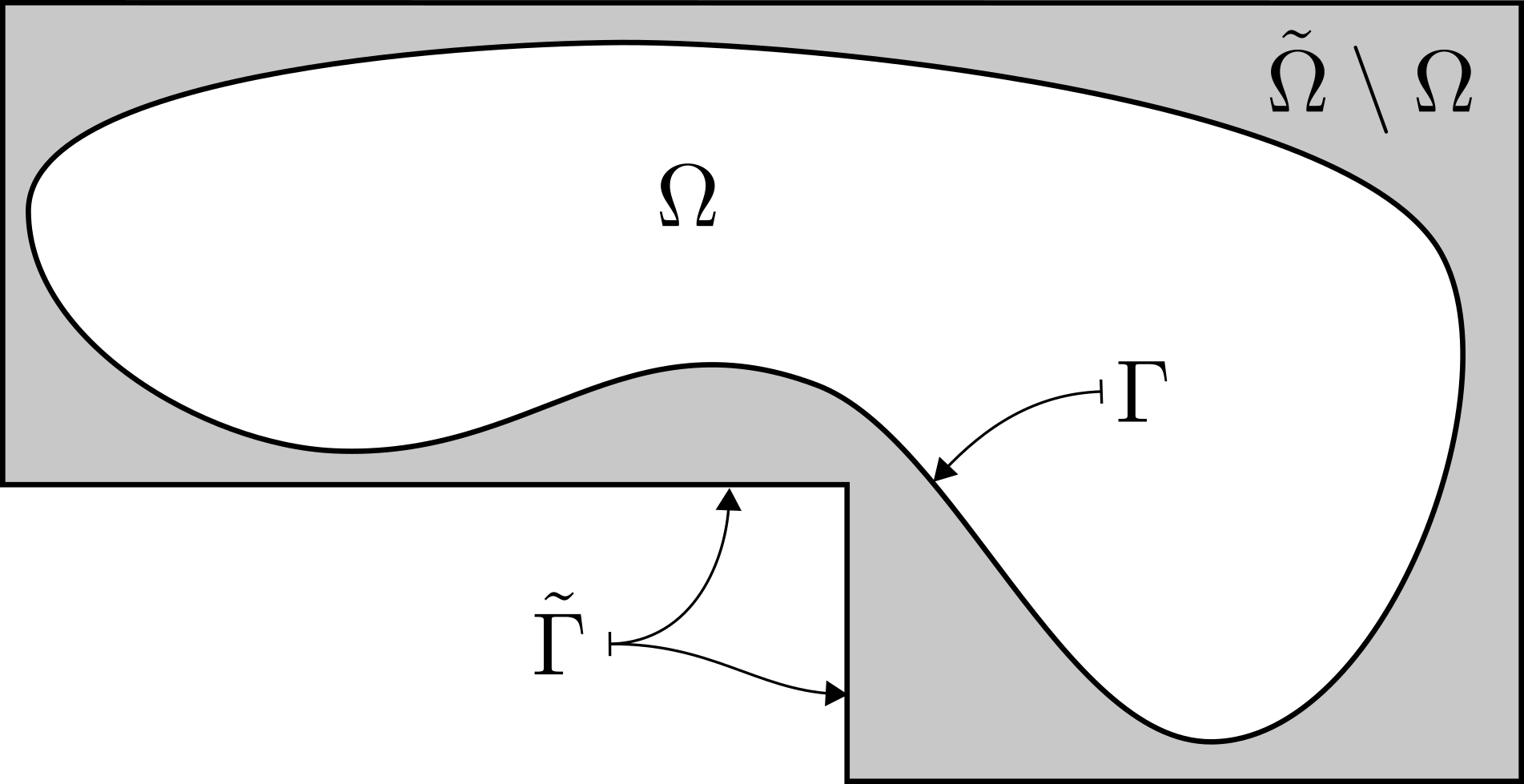}}
    \caption[]{Example domain and boundary for the Laplace equation (left) and the corresponding inverse IBM domain (right): $\tOmega$ consists of the union of the white and gray regions.\label{fig:laplace_geometry}}
  \end{center}
\end{figure}

Having described this generic geometric setup, we now turn to the key question: how do we select the boundary conditions on $\tGamma$ such that the conditions \eqref{eq:laplace_bcs} are satisfied, or 
at least approximately satisfied, on $\Gamma$?   Our tentative answer to this question is to define the following PDE-constrained optimization problem, which is given in variational form:
\begin{equation}\label{eq:III}
\begin{alignedat}{2}
&\inf_{c, u} &\quad &J(u) \equiv \frac{1}{2} \int_{\Gamma} (u - \ubc)^2 \, \diff\Gamma, \\
&\text{s.t.} &\quad
&\bform{u}{v} \equiv \int_{\tOmega} \nabla u \cdot \nabla v \, \diff\Omega = 0,\qquad 
\forall v \in H^{1}_{0}(\tOmega),\\
& & &\RevTwoAdd{c \in \Htrc(\tGamma), \quad\text{and}\quad
u \in \{ v \in H^{1}(\tOmega) \;|\; v|_{\tGamma} = c  \}.}
\end{alignedat}
\end{equation}

\begin{remark}
Notice that the (variational) Laplace equation is posed on the larger domain $\tOmega$.  Furthermore, the boundary condition imposed on $\tGamma$ is defined by $c$, rather than $\ubc$.  The variable $c$ is the control variable used to minimize the cost functional $J(u)$ in order to (indirectly) satisfy the boundary condition on $\Gamma$.
\end{remark}

\RevTwoAdd{
\begin{remark}
One could pose the optimization problem without $c$; that is, we could seek $u \in H^{1}(\tOmega)$ that minimizes $J(u)$ subject to the Laplace equation, $\bform{u}{v} = 0,\; \forall v \in H^{1}(\tOmega)$, where the test function space is the same as the trial space.  We prefer to include the boundary control $c$ because, for a fixed $c$, the PDE constraint is well-posed.  In practice, this means the stiffness matrix is nonsingular, which allows us to leverage existing solution algorithms for PDE-constrained optimization (see Section~\ref{sec:solution}).  Furthermore, without the boundary control, some form of stabilization is necessary to cope with the rank deficient PDE Jacobian.
\end{remark}
}

\begin{remark}
Formulation~\eqref{eq:III} defines a boundary-control problem, and it is the prototype method that we will explore further in the rest of the paper; however, in some cases it is possible to define an inverse IBM using distributed control instead of boundary control.  For example, if $\tOmega$ is defined such that the immersed geometry is bounded and coincides with $\tOmega \setminus \Omega$, then a distributed control can be defined over the region interior to the immersed geometry.  A distributed-control approach is the basis of the fictitious domain method in~\cite{glowinski_he_2011}.
\end{remark}

\subsubsection{Ill-posedness in the context of the Laplace equation}\label{sec:ill-posed}

To establish the well-posedness of \eqref{eq:III} in the sense of Hadamard, we would need to prove that the problem has a unique solution that depends continuously on the data (\eg, $\ubc$); unfortunately, as we show below, there exists inverse IBM problems for which the solution does not depend continuously on the data.  The practical implication of this ill-posedness is that the discretized problem may not be solvable either.

Consider an inverse IBM problem whose target domain is the unit disc, $\Omega =  \{ (r,\theta) \;|\; r < 1, \theta \in [0,2\pi) \}$, and let the boundary condition on $\Gamma$ be given by the following sinusoid:
\begin{equation*}
\ubcn(\theta) = \frac{1}{n} \sin(n \theta),\qquad \forall \; \theta \in [0,2\pi),
\end{equation*}
where $n$ is a positive integer.  \RevOneAdd{The function $\ubcn(\theta)$ represents the data in this context, and we will now define an inverse IBM problem whose solution does not depend continuously on this function}.  Specifically, suppose the larger, enclosing domain is also a disc, but one with radius $R > 1$; thus, $\tOmega = \{ (r,\theta) \;|\; r < R, \theta \in [0,2\pi) \}$.  It is easy to verify that the solution to the Laplace equation on $\tOmega$ that satisfies $u(1,\theta) = \ubcn(\theta)$ is given by
\begin{equation*}
u_n(r,\theta) = \frac{r^n}{n} \sin(n\theta), \qquad \forall \; (r,\theta) \in \tOmega.
\end{equation*}
Consequently, the corresponding (unique) boundary control that solves the inverse IBM problem \eqref{eq:III} is the trace of $u_n(r,\theta)$ on $\tGamma$:
\begin{equation*}
c_n(\theta) = \frac{R^n}{n} \sin(n\theta), \qquad \forall \; \theta \in [0,2\pi).
\end{equation*}
\RevTwoAdd{Now, consider the solution of the inverse IBM as the mapping $T : \Htrc(\Gamma) \rightarrow \Htrc(\tGamma)$ that transforms the boundary data $\ubc$ into $c$, that is $T(\ubc) = c$.  Then this mapping is unbounded in general since 
\begin{equation*}
\| T \| \equiv \sup_{\ubc} \frac{\|T(\ubc)\|_{\Htrc(\tGamma)}}{\|\ubc\|_{\Htrc(\Gamma)}} 
\geq \sup_{n} \frac{\| c_{n} \|_{\Htrc(\tGamma)}}{\|\ubcn\|_{\Htrc({\Gamma})}}
\geq C_{\Omega,\tOmega} \sup_{n} \frac{\| u_n \|_{H^1(\tOmega)}}{\|u_n\|_{H^1(\Omega)}} = \infty,
\end{equation*}
where the constant $C_{\Omega,\tOmega}$ arises from the continuity of the trace operator and its right inverse~\cite[Theorem 1.5.1.2]{grisvard2011elliptic}, and the final result is obtained by substituting $u_n$ into the numerator and denominator.  Informally, the solution $c_n$ can be made arbitrarily large by choosing $n$ sufficiently large even though $\lim_{n\rightarrow \infty} \ubcn = 0$.  We conclude that the solution does not depend continuously on the data and the inverse IBM is not well-posed for the Laplace equation.}

\subsection{Application to the linear advection equation}\label{sec:advect}

We now turn our attention to the linear advection equation, which is an important model problem for inviscid fluid flows and hyperbolic PDEs more generally.  As with the Laplace equation, we will see that the basic inverse IBM formulation is ill-posed when applied to the advection equation.

The linear advection equation on an open, bounded domain $\Omega \subset \mathbb{R}^{N}$, with smooth
boundary $\Gamma$, is given by
\begin{subequations}\label{eq:advect}
\begin{alignat}{2}
  \nabla \cdot(\lambda u) &= 0,
  &\qquad &\forall x \in \Omega, \label{eq:advect_pde} \\
  u &= \ubc,&\qquad &\forall x \in \Gamma^{-}, \label{eq:advect_bcs}
\end{alignat}
\end{subequations}
where $\lambda \in [H^{1}(\Omega)]^{N}$ is the vector-valued velocity field.
Boundary conditions are imposed on the inflow boundary $\Gamma^{-} \equiv \{ x \in \Gamma \;|\; \lambda_n < 0 \}$, where $\lambda_n \equiv \lambda \cdot \hat{n}$ is the velocity component normal to the boundary and $\hat{n}$ is the unit, outward-pointing normal vector on $\Gamma$.

\begin{remark}
We will assume that \eqref{eq:advect} is well-posed for the given data and
geometry.  Furthermore, in order to apply the inverse IBM formulation, we will
assume that the velocity field can be extended to the larger computational
domain $\tOmega$.  That is, we assume that there exists $\tlambda \in
[H^{1}(\tOmega)]^{N}$ such that $\tlambda(x) = \lambda(x)$ for all $x \in
\bar{\Omega}$.
\end{remark}

Next, we formulate an inverse IBM statement, analogous to \eqref{eq:III}, for the linear advection equation.  Specifically, we replace the variational Laplace equation with the variational advection equation in \eqref{eq:III} and obtain the following problem:
\begin{equation}\label{eq:advect_III}
\begin{alignedat}{2}
&\inf_{c, u} &\quad &J(u) \equiv \frac{1}{2} \int_{\Gamma^{-}} (u - \ubc)^2 \, \diff\Gamma, \\
&\text{s.t.} &\quad
  &\b(u,v) \equiv \int_{\tOmega} v \nabla \cdot ( \tlambda u ) \, \diff\Omega = 0,\quad
\forall v \in L^{2}(\tOmega),\\
& & &\RevTwoAdd{%
c \in \Htrc(\tGamma^{-}), \quad\text{and}\quad
u \in \{ v \in L^{2}(\tOmega) \;|\; \nabla \cdot (\tlambda v) \in L^2(\tOmega), \left.v\right|_{\tGamma^{-}} = \left.c\right|_{\tGamma^{-}} \},
}
\end{alignedat}
\end{equation}
where $\tGamma^{-} \equiv \{ x \in \tGamma \;|\; \tlambda_n < 0 \}$ and $\tGamma^{+} \equiv \tGamma \setminus \tGamma^{-}$.
Figure~\ref{fig:advect_III} provides an example of the domains and boundaries that appear in problem \eqref{eq:advect_III}.  We will also refer to this figure when explaining the ill-posedness that affects the inverse IBM in the context of the advection equation.

\begin{figure}[t]
  \begin{center}
    \includegraphics[width=\textwidth]{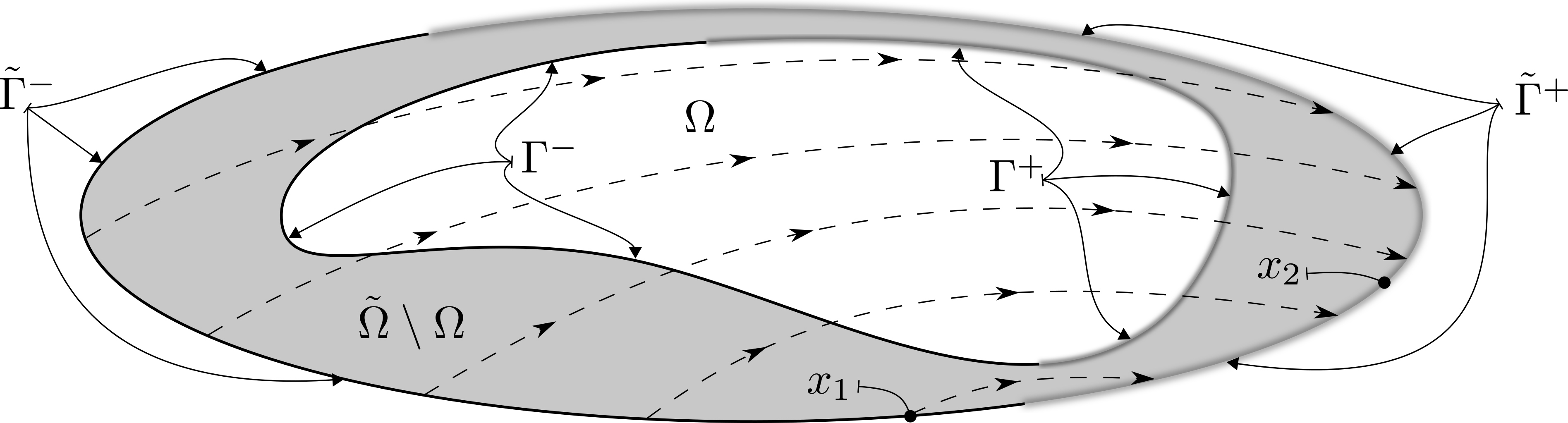}
      \caption[]{Example domains and boundaries for the inverse IBM problem applied to linear advection.\label{fig:advect_III}}
  \end{center}
\end{figure}

\subsubsection{Ill-posedness in the context of the advection equation}

Consider the point labeled $x_1 \in \tGamma^{-}$ in Figure~\ref{fig:advect_III}.  A compact control perturbation centered at $x_1$ determines the value of $u$ along the characteristics downstream of $x_1$ via the boundary condition on $\tGamma^{-}$. However, if the perturbation is sufficiently local, then the corresponding characteristics will not pass through $\Gamma$ and the perturbation will have no impact on the objective.  Consequently, such control perturbations are non-unique in the context of \eqref{eq:advect_III}, \ie the inverse IBM problem is ill-posed for linear advection.

To make matters worse, nonlinear hyperbolic PDEs present an additional difficulty for the proposed inverse IBM formulation.  Namely, the flow field $\tlambda$ itself often depends on the solution, so the definition of $\tGamma^{-}$ and $\tGamma^{+}$ is not necessarily known a priori.  For the discretized version of \eqref{eq:advect_III}, this uncertainty in the set $\tGamma^{-}$ means that the dimension of the boundary control $c \in \Htrc(\tGamma^{-})$ would need to change dynamically during the solution process.  

\begin{remark}
  This practical difficulty is exhibited by the Euler equations of gas dynamics.  For example, if $\Gamma$ is a no-penetration wall in the Euler equations, it will have a known number of incoming and outgoing characteristics (one each); however, the surface $\tGamma$ may require additional incoming/outgoing characteristics in order to satisfy the no-penetration condition.
\end{remark}

Our solution to this practical issue is to define the control on the entire boundary, as was done for the Laplace equation.  On the one hand, this introduces another source of nonuniqueness: the value of the control on $\tGamma^{+}$ --- for example, $x_2 \in \tGamma^{+}$ in Figure~\ref{fig:advect_III} --- has no influence on the state, and, therefore, no influence either satisfying the constraint or minimizing the objective in \eqref{eq:advect_III}.  On the other hand, this form of nonuniqueness is easily eliminated with the regularization we propose in Section~\ref{sec:regularize}.

\section{An investigation of conditioning of the discrete problem}\label{sec:cond}

It is clear that the basic inverse IBM formulation must be modified to be useful in practice; however, before we discuss potential regularizations in Section~\ref{sec:regularize}, we want to better understand the sources of ill-conditioning in the finite-dimensional case.  To this end, we will model the discretization of the Laplace inverse IBM on a particular domain, and we use this model to investigate the conditioning of the finite-dimensional version of \eqref{eq:III}.

\subsection{Description of the model problem}

We consider a model problem similar to the one described in Section~\ref{sec:ill-posed}.  The computational domain is the unit disk,
$\tOmega = \{ (r,\theta) \;|\; r < 1, \theta \in [0,2\pi) \}$, and the immersed  boundary is an ellipse with semi-major and semi-minor lengths $a$ and $b$, respectively.  The model geometry is illustrated in Figure~\ref{fig:model}. We have chosen a circle for $\tOmega$ because it allows us to take advantage of Green's function theory.
  
\begin{figure}[t]
  \begin{center}
    \begin{minipage}{0.49\textwidth}
      \centering 
      \includegraphics[width=\textwidth]{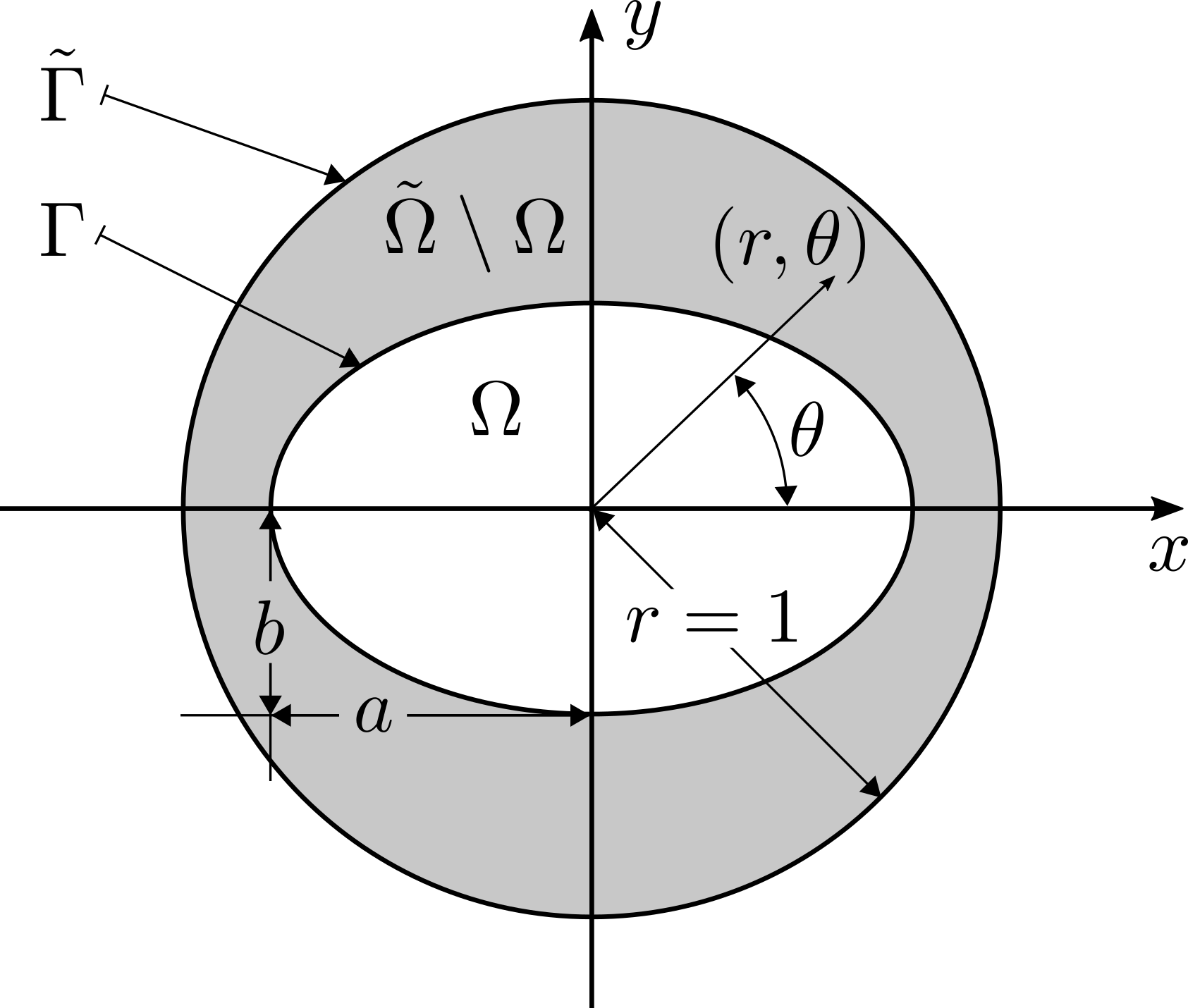}
      \caption[]{The geometry for the model, finite-dimensional Hessian.\label{fig:model}}
    \end{minipage}\hfill
    \begin{minipage}{0.49\textwidth}
      \centering 
      \includegraphics[width=\textwidth]{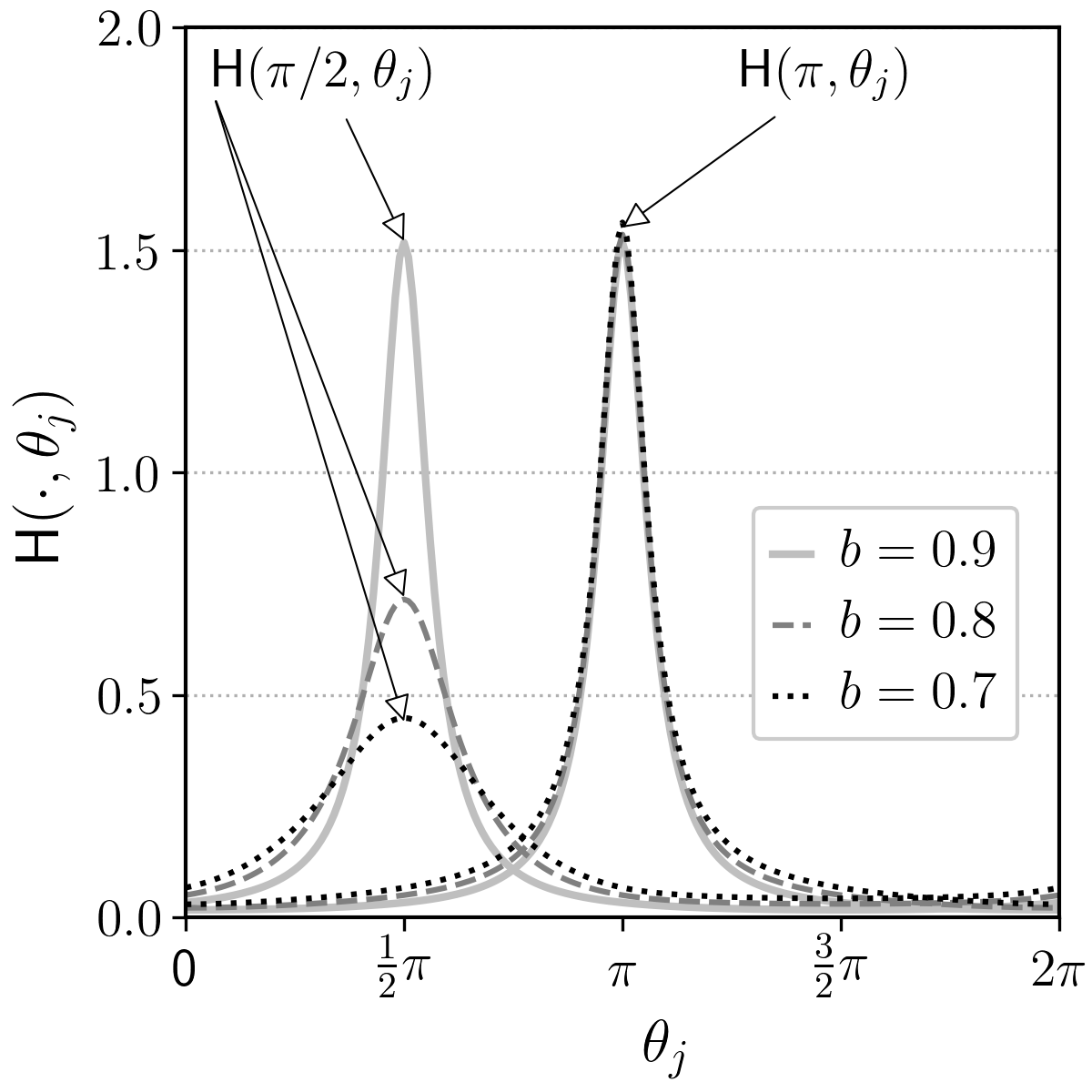}
      \caption[]{The model Hessian for the domains shown in
        Figure~\ref{fig:model}.  Three elliptical shapes and two values of
        $\theta_i$ are shown.  The semi-major axis is $a=0.9$ in all cases. \label{fig:Hess_rows}}
    \end{minipage}
\end{center}
\end{figure}

Our goal is to construct a model for the reduced Hessian of the discretized problem.  To this end, we will first derive the reduced Hessian for the continuous problem, and then discretize the control.

Consider the reduced-space formulation of \eqref{eq:III}, where the state $u$ is considered a function of the control $c$ and the objective is 
\begin{equation*}
  J(c) = J(u(c)) = \frac{1}{2}\int_{\Gamma} u^2(c) \,\diff\Gamma  -  \int_{\Gamma} u(c) \ubc \, \diff\Gamma + J_0,
\end{equation*}
where $J_0 = \frac{1}{2} \int_{\Gamma} \ubc^2 \, \diff\Gamma$.

In order to differentiate $J(c)$ to obtain the Hessian, we will need the derivative of the state with respect to the control.   For the Laplace equation, one can show that the state is a linear and continuous function of the control.  Consequently, the derivative of $u$ with respect to $c$ in the direction $d \in \Htrc(\tGamma)$ is simply $u(d)$; that is, the directional derivative is simply the state obtained by substituting $d$ for $c$.  Furthermore, because $u$ is a linear function of $c$, the second derivative of the state vanishes: $u''(d,e) = 0$, where $d, e \in \Htrc(\tGamma)$ are arbitrary.  Using these facts and differentiating $J$ twice, we find that the Hessian of the continuous-problem objective is
\begin{equation*}
J''(d,e) = \int_{\Gamma} u(d) u(e) \, \diff \Gamma, \qquad \forall\; d, e \in \Htrc(\tGamma).
\end{equation*}

Next, we parameterize $c$ to obtain a finite-dimensional control; the exact form of the parameterization is described later.  Let $c_i \in \mathbb{R}^{K}$ denote the control that is zero everywhere except in the $i$th entry where it is one, \RevTwoAdd{\ie $(c_i)_{k} = \delta_{ik}$}.  Then, taking $d = c_i$ and $e = c_j$ in $J''(d,e)$, we obtain the entries in the finite-dimensional Hessian:
\begin{equation}\label{eq:Hij}
\mat{H}_{ij} \equiv \frac{\partial^2 J}{\partial c_i \partial c_j} = J''(c_i,c_j) = \int_{\Gamma} u(c_i) u(c_j) \,\diff\Gamma, \qquad \forall\; i,j = 1,2,\ldots,K.
\end{equation}

To find an explicit expression for $\mat{H}_{ij}$, we need an explicit expression for $u(c_i)$.  Using the appropriate Green's
function for the disk~\cite[\S 7.1.2]{Polyanin2002handbook}, we obtain the following expression for $u(c(\theta))$,
where $c(\theta)$ is the control value on $\tGamma$:
\begin{equation}\label{eq:udisc}
  u(c) = u(r,\theta) = \frac{1}{2\pi} \int_{0}^{2\pi} \frac{(1 - r^2) c(\tilde{\theta})}{1 + r^2 - 2r \cos(\tilde{\theta} - \theta)} \, d\tilde{\theta}.
\end{equation}

The state given by \eqref{eq:udisc} is for arbitrary $c(\theta)$.  To simplify $u(c_i)$ for the Hessian expression, we use Dirac-delta control variations $c_i$ that are parameterized based on uniformly-spaced points on the unit circle:
\begin{equation*}
  c_i = c(\theta_i) = \delta(\theta - \theta_i),
  \qquad\text{where}\quad \theta_i = 2\pi i/K,\; i=1,2,\ldots,K.
\end{equation*}
Substituting the distributions for $c_i$ and $c_j$ into \eqref{eq:udisc}, and using the resulting $u(c_i)$ and $u(c_j)$ in \eqref{eq:Hij}, we obtain the following expression for $\mat{H}_{ij}$:
\begin{equation}\label{eq:anal_Hess}
\mat{H}_{ij} = \mat{H}(\theta_i, \theta_j) = \frac{1}{(2\pi)^2}\int_{\Gamma}
\frac{(1-r^2)}{[1 + r^2 - 2r \cos(\theta_i - \theta)]}
\frac{(1-r^2)}{[1 + r^2 - 2r \cos(\theta_j - \theta)]} \, \diff\Gamma.
\end{equation}
Note that, while both $r$ and $\theta$ appear in the integrand above, the point
$(r,\theta)$ is constrained to lie on the ellipse $\Gamma$.

Before we use \eqref{eq:anal_Hess} to explore the condition number of the
discrete Hessian, it is worthwhile to plot $\mat{H}_{ij} =
\mat{H}(\theta_i,\theta_j)$ versus $\theta_j$ for fixed $\theta_i$, because these plots will help explain the trends in the condition number.  Two such
$\theta_i$ rows are shown in Figure~\ref{fig:Hess_rows}: one corresponding to
$\theta_i = \pi/2$ and one corresponding to $\theta_i = \pi$.  Each $\mat{H}_{ij}$ entry is evaluated
for three different immersed ellipse boundaries with a fixed semi-major axis of $a=0.9$ and progressively smaller semi-minor axes $b$.  We have approximated the integral in \eqref{eq:anal_Hess} using the trapezoid rule with $n=300$ intervals.

The trend is clear from the figure.  There is little impact on the Hessian row near
the fixed, semi-major axis at $\theta_i = \pi$.  In contrast, decreasing the semi-minor axis
stretches $\mat{H}(\pi/2,\theta)$ horizontally and compresses it vertically.  The
spreading of this row of the  Hessian hints at conditioning issues on the horizon, since relatively ``flat'' rows are difficult to distinguish from one another, \ie the rows are nearly linearly dependent.

\subsection{Circular immersed boundary}

Consider a circular immersed boundary $\Gamma$, with $a=b=r < 1$.  Figure~\ref{fig:cond_vs_offset} plots the condition number of the reduced Hessian $\mat{H}$ versus the Hausdorff distance between $\Gamma$ and $\tGamma$ for three different control-mesh spacings $h = \Delta \theta$.  In the present case, where both $\Gamma$ and $\tGamma$ are circles, the Hausdorff distance is given by $\dH{\Gamma}{\tGamma} = 1 - r$.

For fixed $h$, the condition number grows exponentially with the distance between $\Gamma$ and $\tGamma$.  This growth is due to the diffusing of the analytical Hessian as the immersed boundary gets further from the domain boundary, as described earlier.  Note that, for a fixed $\dH{\Gamma}{\tGamma}$, we also observe a strong dependence on mesh spacing.

Figure~\ref{fig:cond_vs_offset} suggests that it is possible to keep the condition number of $\mat{H}$ fixed as we refine the mesh, provided $\dH{\Gamma}{\tGamma}$ is also reduced.  This relationship is clearly illustrated in Figure~\ref{fig:cond_vs_spacing}, which plots the condition number of $\mat{H}$ versus the ratio $h / \dH{\Gamma}{\tGamma}$.  This figure is important because it shows that different Hausdorff distances essentially collapse onto the same curve, and it is the nondimensional ratio $h/\dH{\Gamma}{\tGamma}$ that determines the condition number.  \RevOneAdd{Our hypothesis is that,} as long as $\lim_{h\rightarrow 0} h/\dH{\Gamma}{\tGamma}$ is bounded, the condition number of the reduced Hessian should also remain bounded.  \RevOneAdd{In particular, this means that the computational domain $\tOmega$ must converge to $\Omega$, in some sense, as the mesh is refined.}

\begin{figure}[t]
  \begin{center}
    \subfigure[$\text{cond}(\mat{H})$ versus Hausdorff distance \label{fig:cond_vs_offset}]{%
      \includegraphics[width=0.45\textwidth]{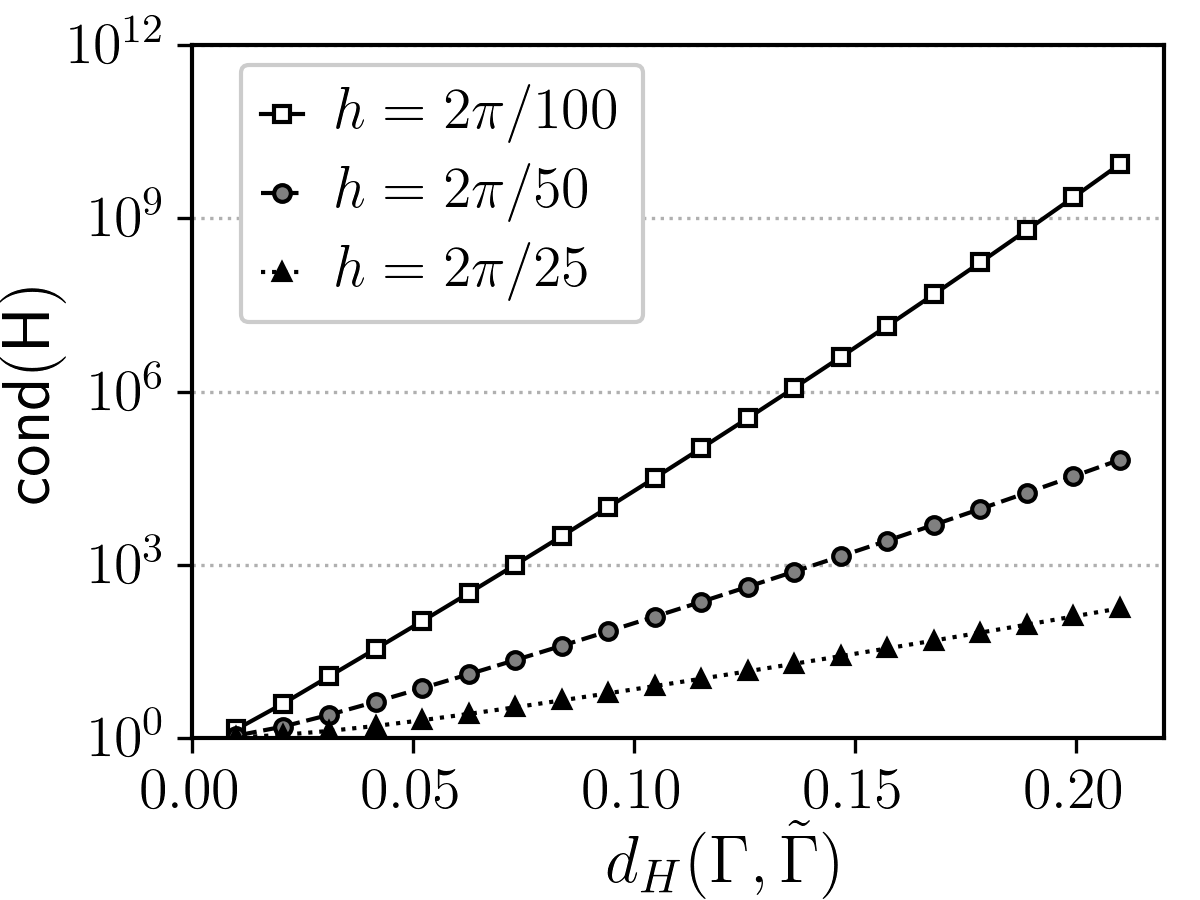}}\hfill
    \subfigure[$\text{cond}(\mat{H})$ versus spacing-distance ratio \label{fig:cond_vs_spacing}]{%
      \includegraphics[width=0.45\textwidth]{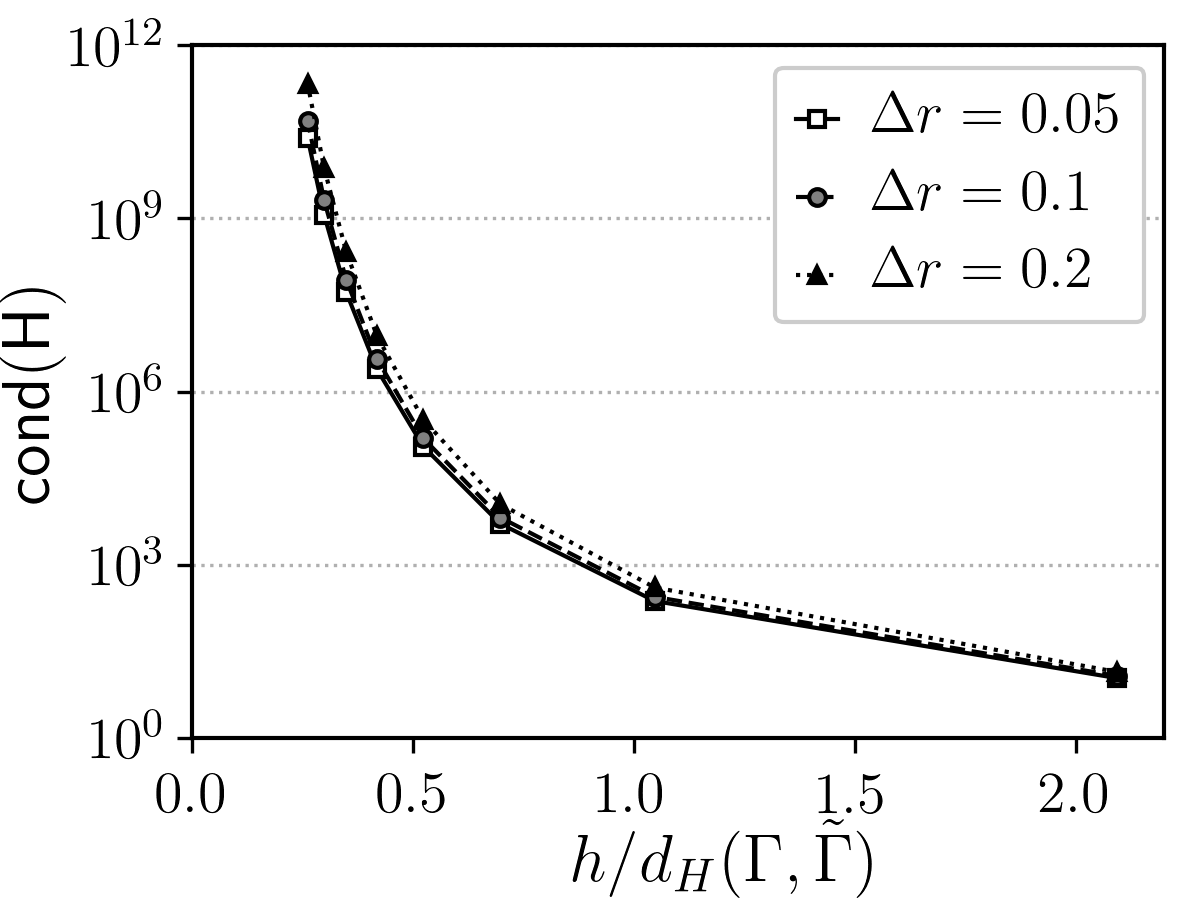}}
    \caption[]{Trends in the reduced Hessian condition number for \RevOneAdd{a circular $\Gamma$ (\ie $a = b$).}\label{fig:cond}}
  \end{center}
\end{figure}

\subsection{Elliptical immersed boundary}

For our next investigation of the model problem, we consider the influence of differing semi-major and semi-minor axis lengths.  Figure~\ref{fig:ellipse_cond_vs_offset} plots the reduced Hessian condition number versus $1-a$ for three different fixed values of $b$.  In all cases the mesh spacing is fixed at $h = 2\pi/50$.  In general, the condition number remains stable as $1-a$ increases until $a < b$, \ie the semi-minor and semi-major axes switch, at which point the exponential growth is similar to that observed for circular $\Gamma$.
This behavior is consistent with the Hausdorff distance determining the condition number for fixed $h$: for $a > b$ the Hausdorff distance is $\dH{\Gamma}{\tGamma} = 1-b$, which is constant for fixed $b$.

Figure~\ref{fig:ellipse_cond_vs_spacing} plots the condition number of $\mat{H}$ versus the nondimensional mesh spacing $h/\dH{\Gamma}{\tGamma}$.  We have included curves corresponding to several different major-minor offset ratios, $(1-b)/(1-a)$.  We observe that $h/\dH{\Gamma}{\tGamma}$ remains the key parameter that determines the condition number.

\begin{figure}[t]
  \begin{center}
    \subfigure[$\text{cond}(\mat{H})$ versus $1-a$, with $h= 2\pi/50$ \label{fig:ellipse_cond_vs_offset}]{%
      \includegraphics[width=0.45\textwidth]{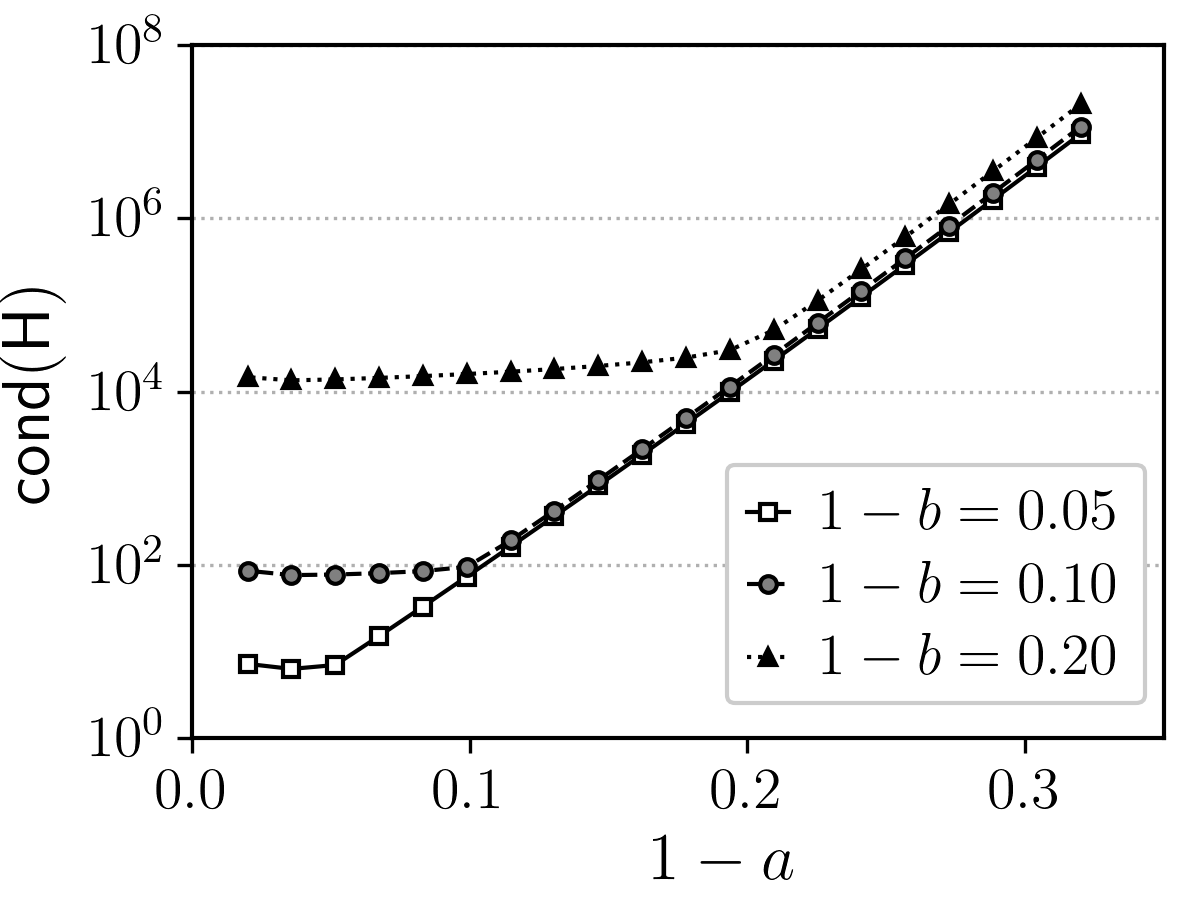}}\hfill
    \subfigure[$\text{cond}(\mat{H})$ versus spacing-distance ratio \label{fig:ellipse_cond_vs_spacing}]{%
      \includegraphics[width=0.45\textwidth]{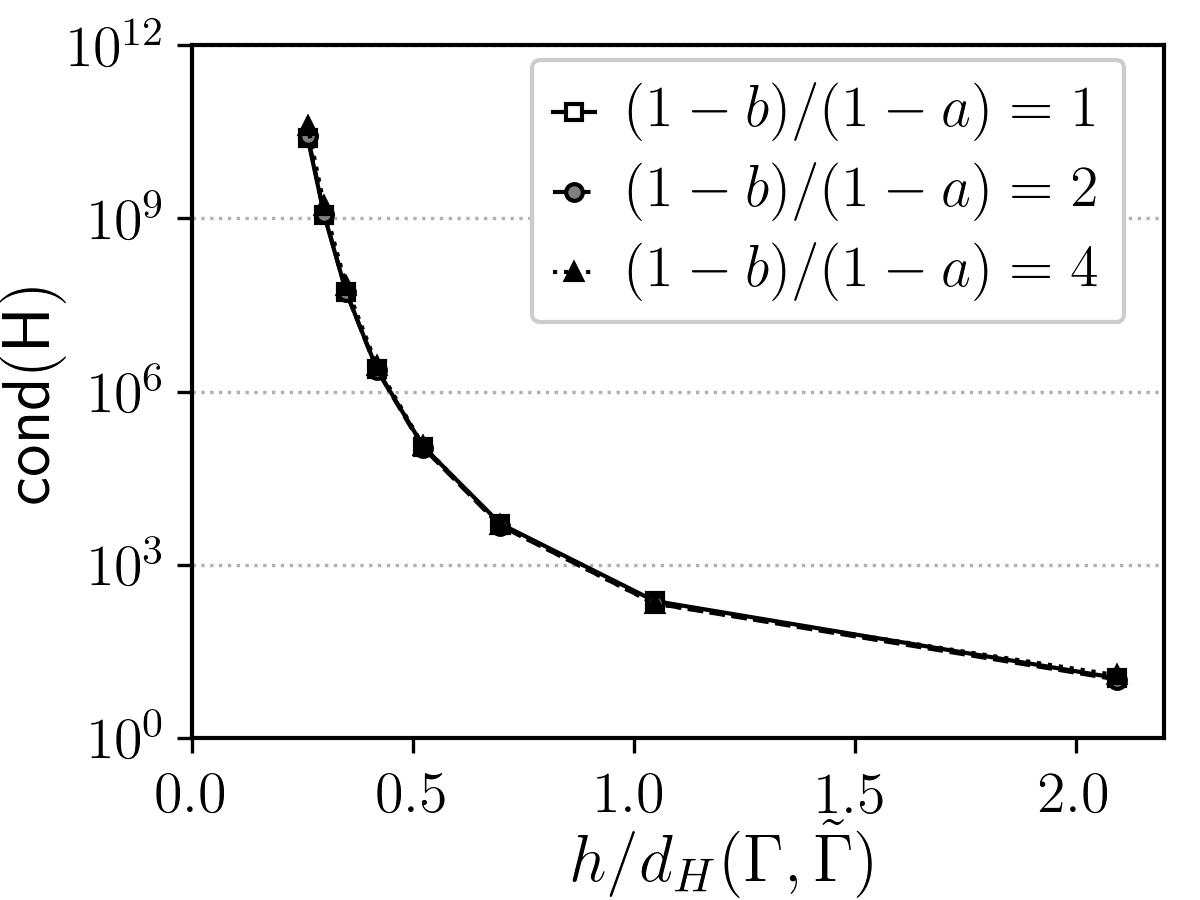}}
    \caption[]{Trends in the reduced Hessian condition number for elliptical
      $\Gamma$.\label{fig:ellipse_cond}}
  \end{center}
\end{figure}

\subsection{Discussion} 

 The results of the above studies suggest that keeping the ratio $h / \dH{\Gamma}{\tGamma}$ bounded is critical to the conditioning of the discretized inverse IBM problem.  This is not surprising, since bounding this ratio implies that $\lim_{h \rightarrow 0} \tGamma = \Gamma$, so the inverse IBM problem becomes essentially equivalent to the direct problem with $c = \ubc$.
 
 However, there are good reasons why bounding $h / \dH{\Gamma}{\tGamma}$ alone may not be sufficient in practice.
\begin{itemize}
\item The model problem does not reflect the behavior of more difficult geometries, \eg, non-convex domains.  Indeed, in the results section we consider one problem for which the basic formulation of the inverse IBM works (the reduced Hessian is invertible), and one problem for which the reduced Hessian is singular. 

\item For a desired mesh resolution $h$, it may be challenging to generate a mesh for which the distance between $\Gamma$ and $\tGamma$ is sufficiently small to avoid large condition numbers; indeed, in the limit as $\dH{\Gamma}{\tGamma} \rightarrow 0$, we are back to generating a conformal mesh.
 
 \item The non-unique solutions that the inverse IBM admits for linear advection are not excluded by bounding $h / \dH{\Gamma}{\tGamma}$.  Thus, without further intervention, the discretized IBM will produce a singular system when applied to the linear advection equation.
 \end{itemize}
 For the reasons listed above, we need to consider some form of regularization for the inverse IBM.

\section{Regularization}\label{sec:regularize}

We have seen that the basic formulation of the inverse IBM is ill-posed for both the Laplace and linear advection equations.  Furthermore, while bounding the ratio $h/\dH{\Gamma}{\tGamma}$ seems to help in some cases, this strategy does not suffice on its own for all problems.  Therefore, the basic formulation must be modified if it is to be useful in practice.

Regularized formulations are commonly used to ensure inverse problems are both well-posed and well-conditioned, so this is an obvious strategy to examine for the inverse IBM.  In this section, we briefly review two popular regularizations used for inverse problems and argue why they are poorly suited for the inverse IBM.  We then introduce a novel regularization for the inverse IBM that is well suited for discretizations that use weakly imposed boundary conditions.

\subsection{Review of regularization approaches for inverse problems}

The popular regularizations that we review here are the Tikhonov and total variation (TV) diminishing regularizations.  Both methods modify the objective function $J(u)$ by adding a (scaled) convex term, which we denote by $S$. Thus, the modified objective is given by
\begin{equation}\label{eq:Jreg}
\Jreg(u,c) = J(u) + \alpha S(u,c),
\end{equation}
where $\alpha > 0$ is a regularization parameter.

\begin{remark}
The regularized optimization problem based on $\Jreg$ corresponds to a bi-objective problem in which $\alpha$ implicitly determines a particular Pareto optimal solution.  Relatively small values of $\alpha$ select solutions that emphasize the boundary-condition accuracy at the cost of ill-conditioning, while large $\alpha$ produce better conditioned problems whose solutions may have significant mismatch in the boundary condition.
\end{remark}

In the case of Tikhonov regularization applied to the inverse IBM, we have
\begin{equation}\label{eq:Tik_reg}
S_{\text{Tik}}(c) \equiv \frac{1}{2} \int_{\tGamma} (c-c_0)^2 \, \diff\Gamma,
\end{equation}
where $c_0 \in \Htrc(\tGamma)$ is a constant function that can be used to reduce the detrimental effect that the regularization has on solution accuracy.  The TV-diminishing regularization takes the form
\begin{equation*}
S_{\text{TV}}(c) \equiv \int_{\tGamma} \sqrt{ \nabla_{\tGamma} c \cdot \nabla_{\tGamma} c + \epsilon} \, \diff\Gamma,
\end{equation*}
where $c \in H^{\sss\frac{3}{2}}(\tGamma)$ and $\nabla_{\tGamma} c$ denotes the gradient of the control with respect to a parameterization of $\tGamma$.  The constant $\epsilon > 0$ is a small parameter that is included to ensure the regularization is differentiable at points where $\nabla_{\tGamma} c= 0$.  TV-based regularization was originally introduced in digital image processing for noise removal~\cite{RUDIN1992TVD}. 

The Tikhonov term $S_{\text{Tik}}$ is quadratic, strongly convex, and it tends to produce smooth solutions.  In contrast, the TV-diminishing term tends to produce piecewise constant solutions that may arise, for example, in the presence of discontinuous physical properties.  Note that the TV regularization is not strongly convex --- any constant $c$ minimizes $S_{\text{TV}}$ --- so it may not be sufficient to guarantee well-posedness, in general.  Furthermore, the TV regularization is not quadratic, so iterative solution methods are required even for linear state equations.

For both regularizations, choosing an appropriate value for the parameter $\alpha$ can be challenging: it should be sufficiently large to stabilize the problem, but not too large to significantly impact solution accuracy; see~\cite{Troltzsch2009} for an error estimation of the solution with respect to $\alpha$.  Some common techniques to select $\alpha$ include the $L$-curve method~\cite{Hansen_Lcurve} and the Morozov discrepancy principle~\cite{Scherzer1993}.

In the case of the inverse IBM, selecting the parameter $\alpha$ is further complicated by the requirements of solution convergence as $h\rightarrow 0$.  To explain this complication, we consider a conforming finite-element discretization of the Laplace inverse IBM problem \eqref{eq:III} with Tikhonov regularization: find $c_h \in \Vh(\tGamma)$ and $u_h \in \Vh(\tOmega)$ that satisfy
\begin{equation}\label{eq:discrete_III}
\begin{alignedat}{2}
&\min_{c_h, u_h} &\quad &\Jreg(u_h,c_h) \equiv \frac{1}{2} \int_{\Gamma} (u_h - \ubc)^2 \, \diff\Gamma + \frac{\alpha}{2} \int_{\tGamma} c_h^2\, \diff\Gamma  \\
&\text{s.t.} &\quad
&\bhform{u_h}{v_h}  = \int_{\tOmega} \nabla u_h \cdot \nabla v_h \, \diff\Omega = 0, \qquad \forall v_h \in \Vh'(\tOmega),
\end{alignedat}
\end{equation}
where $\Vh(\tGamma)$, $\Vh(\tOmega)$ and $\Vh'(\tOmega)$ are appropriate finite-element spaces.

\begin{theorem}\label{thm:alpha}
Assume that the data and geometry in problem \eqref{eq:III} are such that a unique solution exists, \eg, the model problem in Section \ref{sec:cond}, and denote this unique solution by $c \in \Htrc(\tGamma)$ and $u \in H^{1}(\tOmega)$.  Let $c_h \in \Vh(\tGamma)$ and $u_h \in \Vh(\tOmega)$ be the solution to \eqref{eq:discrete_III}, and let $h$ denote the nominal element size.  Suppose the $L^2$ solution error on $\Gamma$ satisfies $\|u_h - u\|_{\Gamma} = \text{O}(h^{p+1})$ for some integer $p \geq 1$ and sufficiently small $h$.  Furthermore, suppose the discrete control is bounded below, $\| c_h \|_{\tGamma} > B > 0$, for sufficiently small $h$.  Then there exists $h^{\star} > 0$ such that the regularization parameter satisfies
\begin{equation*}
\alpha \leq M h^{p+1}, \qquad \forall h \leq h^{\star},
\end{equation*}
where $M >0$ is independent of $h$.
\end{theorem}

\begin{proof}
The first-order optimality conditions for \eqref{eq:discrete_III} in the reduced space are 
\begin{equation*}
\Jreg'(d_h) = \int_{\Gamma} (u_h(c_h) - \ubc) u_h(d_h) \, \diff\Gamma + \alpha \int_{\tGamma} c_h d_h \, \diff\Gamma = 0,
\qquad \forall d_h \in \Vh(\tGamma),
\end{equation*}
where, as in the continuous case, the directional derivative $(\nabla_{c_h} u_h) \cdot d_h = u_h(d_h)$ follows from the linear dependence of $u_h$ on $c_h$.  In particular, for $d_h = c_h$, we have
\begin{equation*}
\Jreg'(c_h) = \int_{\Gamma} (u_h(c_h) - \ubc) u_h(c_h) \, \diff\Gamma + \alpha \int_{\tGamma} c_h^2 \, \diff\Gamma = 0.
\end{equation*}
Rearranging the above equality and applying Cauchy-Schwarz, we find
\begin{equation*}
\alpha \| c_{h} \|_{\tGamma}^2 \leq \| u_h(c_h) - \ubc \|_{\Gamma} \| u_h(c_h) \|_{\Gamma}.
\end{equation*}
The assumptions on the solution error and $\|c_{h}\|_{\tGamma}^2$ imply that there exists an $h^{\star} > 0$ such that both $\|c_{h}\|_{\tGamma}^2 > B^2$ and $\| u_h(c_h) - \ubc \|_{\Gamma} \| u_h(c_h) \|_{\Gamma}  \leq \Lambda h^{p+1} \| \ubc \|_{\Gamma}$, for some $\Lambda > 0$ and all $h < h^{\star}$.  Using these bounds in the above inequality we arrive at
\begin{equation*}
\alpha \leq \frac{\| u_h(c_h) - \ubc \|_{\Gamma} \| u_h(c_h) \|_{\Gamma}}{\| c_{h} \|_{\tGamma}^2}
\leq \frac{\Lambda h^{p+1} \| \ubc \|_{\Gamma}}{B^2},\qquad \forall h < h^{\star}.
\end{equation*}
The result follows with $M = \Lambda \| \ubc \|_{\Gamma}/B^2$.
\end{proof}

The implication of Theorem~\ref{thm:alpha} is that, if we have optimal solution error, then the positive effect that Tikhonov regularization has on conditioning must necessarily vanish as $h\rightarrow 0$.  In other words, the ill-conditioning will return as the mesh is refined.  A similar conclusion can be drawn with TV-based regularization.  This motivates our search for a regularization that is compatible with high-order accuracy, but whose Hessian does not vanish with mesh refinement.

\begin{remark}
  Theorem~\ref{thm:alpha} gives a necessary but not sufficient condition on $\alpha$.  Indeed, our numerical experiments (not reported here) indicate that in some cases optimal solution accuracy is achieved only if $\alpha = \text{O}(h^{2p})$.
\end{remark}

\subsection{Regularization via boundary-condition penalization}

Our proposed regularization for the inverse IBM is inspired by weakly imposed boundary conditions.  To motivate the regularization, recall our discussion of the linear advection problem in Section~\ref{sec:advect}.  If the boundary control is defined on all of $\tGamma$, we explained why a variation in $c$ on the downwind boundary $\tGamma^{+}$ does not influence the state.  While the control $c$ is intended to define the boundary value, we can reverse the causation on $\tGamma^{+}$ and insist that $u$ determines $c$.  More generally, for weakly imposed boundary conditions, we can seek a solution for which the difference between the \RevTwoAdd{discretized} state and control is small \emph{on the entire boundary $\tGamma$}.  That is, we consider a \RevTwoAdd{discrete} regularization of the form
\RevTwoAdd{%
\begin{equation}\label{eq:our_reg}
S(u_h,c_h) = \frac{1}{2} \int_{\tGamma} (u_h - c_h)^2 \, \diff\Gamma.
\end{equation}
}

\RevTwoAdd{
\begin{remark}
The regularization \eqref{eq:our_reg} is not useful for the continuous inverse IBM problem, since $c = \left.u\right|_{\tGamma}$ in that case.  However, the term is nonzero when the problem is discretized with weakly imposed boundary conditions, because $c_h \neq \left.u_h\right|_{\tGamma}$, in general.
\end{remark}
}

\begin{remark}
The regularization \eqref{eq:our_reg} is analogous to a Tikhonov regularization \eqref{eq:Tik_reg} with $c_0 = u$.
\end{remark}

\ignore{
While the regularization \eqref{eq:our_reg} is compatible with strongly imposed boundary conditions, it is better suited for weakly imposed boundary conditions.  For example, it can be used to eliminate ill-posedness due to defining $c$ on all of $\tGamma$ when the inverse IBM is applied to the linear advection equation; however, the regularization term has no further benefit when the boundary conditions are strongly imposed.  In contrast, we have found that \eqref{eq:our_reg} improves conditioning significantly for both the Laplace and linear advection equations when weakly imposed boundary conditions are used; see Section~\ref{sec:results}.
}

Unlike Tikhonov and TV-diminishing regularizations, the discretized version of \eqref{eq:our_reg} tends to zero as $h\rightarrow 0$.  Consequently, any (constant) choice of the parameter $\alpha$ in \eqref{eq:Jreg} is compatible with an optimal solution convergence rate.
For example, consider the Laplace equation under the assumptions in Theorem~\ref{thm:alpha} and suppose the optimal convergence rate is $\text{O}(h^{p+1})$.  Using the regularization \eqref{eq:our_reg}, the first-order optimality conditions in the reduced space imply (see the proof of Theorem~\ref{thm:alpha} for further details)
\begin{equation*}
\int_{\Gamma} (u_h(c_h) - \ubc) u_h(d_h) \, \diff\Gamma + \alpha \int_{\tGamma} (u_h(c_h) - c_h) (u_h(d_h) - d_h) \, \diff\Gamma = 0,
\qquad \forall d_h \in \Vh(\tGamma).
\end{equation*}
Previously, with Tikhonov regularization, it was necessary (but not sufficient) to have $\alpha = \text{O}(h^{p+1})$ in order to satisfy the first-order optimality conditions as $h \rightarrow 0$.  In contrast, the proposed regularization is consistent with the above, since both terms in the condition asymptotically approach zero.  In particular, assuming optimal convergence of the state to the boundary value on $\tGamma$, we have $|\int_{\tGamma} (u_h - c_h)^2 \, \diff\Gamma | = \text{O}(h^{2p+2})$.

\begin{remark}
Based on the above discussion, any constant choice for $\alpha$ is consistent with optimal solution convergence.
We adopt $\alpha = 1$ for all subsequent numerical experiments. 
\end{remark}

\RevTwoAdd{
The results, presented in the following section, confirm that the regularization \eqref{eq:our_reg} improves the conditioning of the discretized inverse IBM problem; however, \emph{this regularization does not prevent ill-posedness as $h\rightarrow 0$}.  To ensure a well-posed problem our hypothesis remains as before: the ratio $h / \dH{\Gamma}{\tGamma}$ must be bounded as the mesh is refined.  Thus, the regularization should be used in conjunction with a computational domain $\tOmega$ that converges to $\Omega$.
}

\section{A numerical investigation of accuracy and conditioning}\label{sec:results}

The purpose of the following numerical experiments is investigate i) the conditioning and ii) solution accuracy of the inverse IBM.  For these studies we consider the steady, constant-coefficient, advection-diffusion equation:
\begin{equation}\label{eq:adv-diff}
\begin{aligned}
\nabla \cdot (\lambda u - \mu\nabla u) &= f,
& \qquad &\forall x \in \Omega,\\
u &= \ubc,& \qquad &\forall x \in \Gamma',
\end{aligned}
\end{equation}
where $\lambda \in \mathbb{R}^2$ is the advection velocity and $\mu \in \mathbb{R}$ is the non-negative diffusion coefficient. 

In the numerical experiments, three sets of coefficients are chosen to model different physics.  These sets are as follows:
\begin{itemize}
  \item $\lambda = [1,1]^T$ and $\mu=0$ for a pure advection problem;
  \item $\lambda = [0,0]^T$, $\mu=1$ for a Poisson-type diffusion problem;
  \item $\lambda = [1,1]^T$, $\mu=10^{-2}$ for an advection-diffusion problem.
\end{itemize}
The boundary conditions are imposed on $\Gamma' \equiv \Gamma$ for pure diffusion and advection-diffusion, and they are imposed on $\Gamma' \equiv \Gamma^{-}$ for pure advection.

\subsection{Discretization of the inverse IBM formulation}

\subsubsection{Discretization of the advection-diffusion equation}

We use a standard discontinuous Galerkin (DG) finite-element method to discretize the advection-diffusion equation~\eqref{eq:adv-diff}.  The symmetric interior penalty Galerkin (SIPG) method~\cite{Arnold1982interior, Shahbazi2005explicit} is used to discretize the diffusion term while standard upwinding is used for the advection part; see, for example, \cite{Houston2002}.  For completeness, a detailed description of the discretization is provided in Appendix~\ref{app:disc}.

\begin{remark}
While we have adopted a DG finite-element method here, we emphasize that the proposed inverse IBM formulation is agnostic to the choice of discretization with the exception of the regularization \eqref{eq:our_reg}, which requires weakly imposed boundary conditions.
\end{remark}

\subsubsection{Discretization of the objective function}

Let us first consider the discretization of $J(u) = \int_{\Gamma'} (u - \ubc)^2\, \diff\Gamma$.  We uniformly divide the boundary contour $\Gamma$ into $n_{\Gamma}$ line segments.  On each line segment we introduce Gauss quadrature points, and we construct an interpolation to each quadrature point based on the finite element that contains the point.  We can then use the interpolated values to evaluate $u_h$ on $\Gamma$.   The value of $\ubc$ at the quadrature points can be determined from the manufactured solution, which is given by either~\eqref{eq:man_sln} or \eqref{eq:nonsmooth_soln}.

For the following results, we employ a quadrature rule that is exact for polynomials of degree $p$ for each segment on $\Gamma$, where $p$ is the degree of the Lagrange bases used in the discretization~\eqref{eq:adv-diff_DG}. In addition, we consider different values of $h_{\Gamma}$ when studying the condition number of the reduced Hessian; however, in the solution-convergence study, we choose $n_\Gamma$ such that $h_{\Gamma}/h \approx 0.5$.

\begin{remark}
The size of the line segments, $h_\Gamma$, as well as the number of quadrature points on each line segment, can be varied independently of the finite-element discretization; however, as we shall see, there is a limit to how large the ratio $h_{\Gamma}/h$ can be made, for a given quadrature rule, before it impacts the conditioning of the discrete system. 
\end{remark}

Finally, the regularization term $S(u_h,c_h) = \int_{\tGamma} (u_h - c_h)^2 \, \diff\Gamma$ is evaluated directly using the trace of $u_h$ and the value of $c_h$ on the edges $e \in \tGamma_h$. 

\subsubsection{Definition of the expanded domain $\tOmega$}\label{sec:tOmega}

The domain $\tOmega$ is determined using the following process.  We begin by generating a rectangular background mesh composed of uniform triangles that is of sufficient size to cover the target domain $\Omega$.  Then $\tOmega$ is defined as the set of all elements that are either interior to $\Omega$ or are intersected by $\Gamma$.  To illustrate, Figure~\ref{fig:mesh_disk} shows the coarsest two meshes for a unit disk domain and the corresponding $\tOmega$. The nominal element size $h$ is defined as the square root of the triangle areas.

\begin{figure}[tbp]
  \centering
  \subfigure[coarsest mesh, $h=H$ \label{fig:mesh1}]{%
    \includegraphics[width=0.45\textwidth]{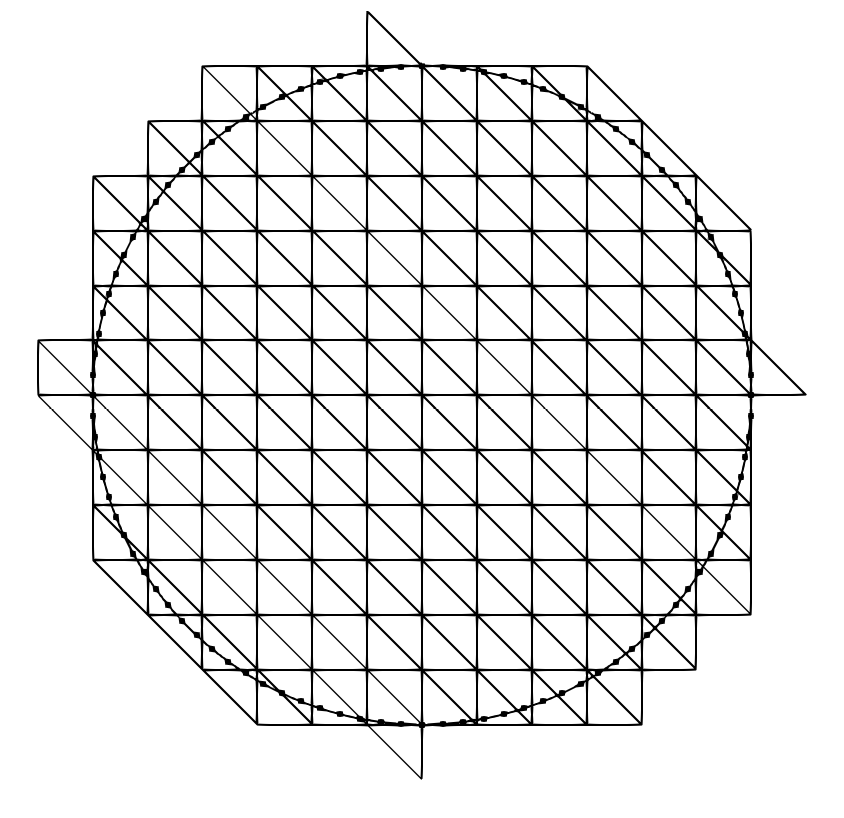}}\hfill
  \subfigure[next coarsest mesh, $h=H/2$ \label{fig:mesh2}]{%
    \includegraphics[width=0.45\textwidth]{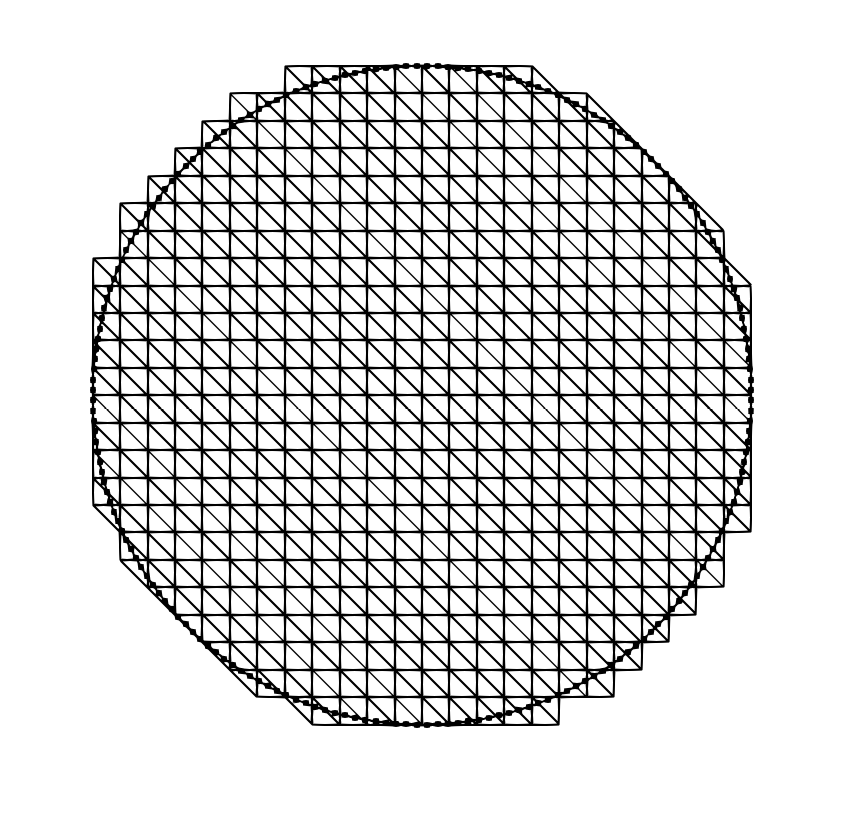}}
  \caption[]{The coarsest two meshes for the unit disk domain used for the numerical experiments.  The
    small black dots on circle denote the quadrature locations where the
    boundary-condition mismatch term in the objective is
    evaluated.  \label{fig:mesh_disk}}
\end{figure}

\begin{remark}
The proposed construction for $\tOmega$ ensures that the ratio $h/\dH{\Gamma}{\tGamma}$ remains bounded as $h \rightarrow 0$, which is what we want in practice.  On the other hand, the construction makes it impossible to vary the Hausdorff distance $\dH{\Gamma}{\tGamma}$ independently of $h$, as we did for the model problem in Section~\ref{sec:cond}; however, we can still investigate the impact of the ratio of $\dH{\Gamma}{\tGamma}$ on the spacing of the degrees of freedom by varying the finite-element polynomial degree $p$.  We will present results of such a study in Section~\ref{sec:cond_results}.
\end{remark}

\begin{remark}
Recall the ill-posedness that affects the inverse IBM applied to the linear advection equation, \eg point $x_1$ in Figure~\ref{fig:advect_III}.  We hypothesize that the proposed construction of $\tOmega$ helps avoid ill-conditioning produced by such points, since any element containing these points is necessarily coupled to elements interior to the domain $\Omega$; however, further study is necessary to confirm this claim.
\end{remark}

\subsection{Solution of the discretized inverse IBM system}\label{sec:solution}

We now briefly discuss the solution of the discretized inverse IBM formulation, as well as the construction of the reduced Hessian matrix used to study the conditioning of the system.

Let $\tOmega_h$ be the triangularization of $\tOmega$, let $\{\Phi_{i}\}_{i=1}^{n}$ denote the basis for the discrete state space $\Vhp(\tOmega_h)$, and let $\{\phi_{i}\}_{i=1}^{m}$ be the basis for the discretized control space $\Vhp(\tGamma_h)$; see Appendix~\ref{app:disc} for definitions of $\Vhp(\tOmega_h)$ and $\Vhp(\tGamma_h)$.  The finite-element state and control can be represented as
\begin{equation*}
u_h(x) = \sum_{i=1}^{n} (\bar{u}_{h})_i \Phi_i(x), \quad\text{and}\quad 
c_h(x) = \sum_{i=1}^{m} (\bar{c}_{h})_i \phi_i(x),
\end{equation*}
respectively, where $\uh \in \mathbb{R}^{n}$ and $\bar{c}_{h} \in \mathbb{R}^{m}$ are the nodal coefficients.  Substituting these expressions for $u_h$ and $c_h$ into the advection-diffusion trilinear form \eqref{eq:adv-diff_DG} and the discretized objective, the discretized inverse IBM problem can be expressed as the following finite-dimensional optimization problem:
\begin{equation}\label{eq:III_DG}
\begin{alignedat}{2}
&\min_{\ch, \uh} &\quad &J_h(u_h,c_h) = \frac{1}{2} \ch^T \Hcc \ch - \ch^T \Hcu \uh + \frac{1}{2} \uh^T \Huu \uh - \bvech^T \uh + J_{0}, \\
&\text{s.t.} &\quad &\b_h(u_h, c_h, v_h) = 
\vh^T\left(\Au \uh + \Ac \ch - \fh\right) = 0,\qquad 
\forall \vh \in \mathbb{R}^{n}.
\end{alignedat}
\end{equation}


The optimization problem~\eqref{eq:III_DG} is a quadratic program since it has a quadratic objective and a linear constraint.  Furthermore, one can show that the objective is convex in $\ch$ and $\uh$.  Consequently, the solution to \eqref{eq:III_DG} is equivalent to the solution to the following saddle-point problem:
\begin{equation}\label{eq:saddle}
\begin{bmatrix}
\phantom{-}\Huu & -\Hcu^T &\Au^T \\
-\Hcu& \phantom{-}\Hcc & \Ac^T \\
\phantom{-}\Au & \phantom{-}\Ac  & 0
\end{bmatrix}
\begin{bmatrix} 
\uh \\ \ch \\ \psih
\end{bmatrix}
= 
\begin{bmatrix}
\bvech \\ 0 \\ \fh
\end{bmatrix},
\end{equation}
where $\psih \in \mathbb{R}^{n}$ are the Lagrange multipliers associated with the constraint; in the present context, the multipliers are equivalent to the nodal coefficients of the adjoint.  We solve the linear system \eqref{eq:saddle} using a sparse direct solver.

\begin{remark}
In practice, the efficient solution of \eqref{eq:III_DG} or \eqref{eq:saddle} would require more sophisticated algorithms, such as sparse iterative solvers with special-purpose preconditioners and Newton's method for nonlinear PDEs.  The application of such algorithms to \eqref{eq:III_DG} is the subject of ongoing research --- see, for example, the preliminary results in \cite{Yan2018immersed} --- but is beyond the scope of the current work.  
\end{remark}

The optimization statement \eqref{eq:III_DG} and linear system \eqref{eq:saddle} are full-space formulations, since the control, state, and multipliers are solved simultaneously; see, for example, \cite{LNKS2006, borzi:2011, akcelik:2006}.  Alternatively, we can eliminate the state and multipliers to arrive at the following reduced-space optimization statement:
\begin{equation*}
\min_{\ch} \quad J_h(\ch) = \frac{1}{2}\ch^T \Hz \ch + g_c^T\ch + J_{0,c}
\end{equation*}
where $J_{0,c}$ is a constant, $g_c \in \mathbb{R}^{m}$ is the reduced gradient and
\begin{equation}\label{eq:Hz}
\Hz = \Ac^T\Au^{-T}\Huu\Au^{-1}\Ac
- \Ac^T\Au^{-T}\Hcu^T - \Hcu\Au^{-1}\Ac + \Hcc
\end{equation}
is the reduced Hessian.  We will study the conditioning of the reduced Hessian in the following sections, since this provides an indication of the difficulty of solving \eqref{eq:saddle} using iterative methods.  Furthermore, we can use $\Hz$ to verify our conclusions regarding the model problem from Section~\ref{sec:cond}.

\subsection{Conditioning of the reduced Hessian}\label{sec:cond_results}

This section investigates how the proposed regularization helps the conditioning of $\Hz$ on both convex and nonconvex geometries.  We also consider the impact of the ratio $h/h_\Gamma$.

\subsubsection{Poisson equation on the unit-disk domain}\label{sec:cond_laplace}

We begin our investigation of the conditioning of $\Hz$, defined by~\eqref{eq:Hz},
with the Poisson-type diffusion case, that is, with $\lambda = [0,0]^T$ and
$\mu=1$ in \eqref{eq:adv-diff}.  The problem domain is the unit disk, $\Omega =
\{ x \in \mathbb{R}^2 \;|\; \|x\|_2 < 1 \}$.

Recall that the immersed boundary is discretized into line $n_{\Gamma}$ segments of length $h_{\Gamma} $, and that $h_{\Gamma}$ is independent of the volume mesh size $h$.  
However, if there are $m$ control degrees of freedom, we expect that $n_{\Gamma} \gtrapprox m$ is necessary for the reduced Hessian to be non-singular.  This intuition is supported by the results in Figure~\ref{fig:cond_vs_hgamma_noreg}, which plots the condition number of $\Hz$ for the basic (unregularized) inverse IBM versus the ratio $h_{\Gamma}/h$. 
Results for a range of element sizes $h$ are shown, with $p=1$ elements adopted in all cases.  We see that the condition number is relatively constant until a threshold of $h_{\Gamma}/h \approx 0.5$ is reached, at which point $\Hz$ becomes singular to working precision.  

Figure~\ref{fig:cond_vs_hgamma_reg} shows the analogous results for the regularized objective.  The figure highlights two benefits of the regularization for diffusive PDEs.  First, rather than abruptly becoming singular, the condition number of $\Hz$ gradually increases when $h_{\Gamma}/h$ passes a threshold of approximately one.  Second, for values of $h_{\Gamma}/h$ below this threshold, the condition number of $\Hz$ is approximately two orders of magnitude smaller than when no regularization is used.

\begin{figure}[t]
  \begin{center}
    \subfigure[no regularization \label{fig:cond_vs_hgamma_noreg}]{%
      \includegraphics[width=0.45\textwidth]{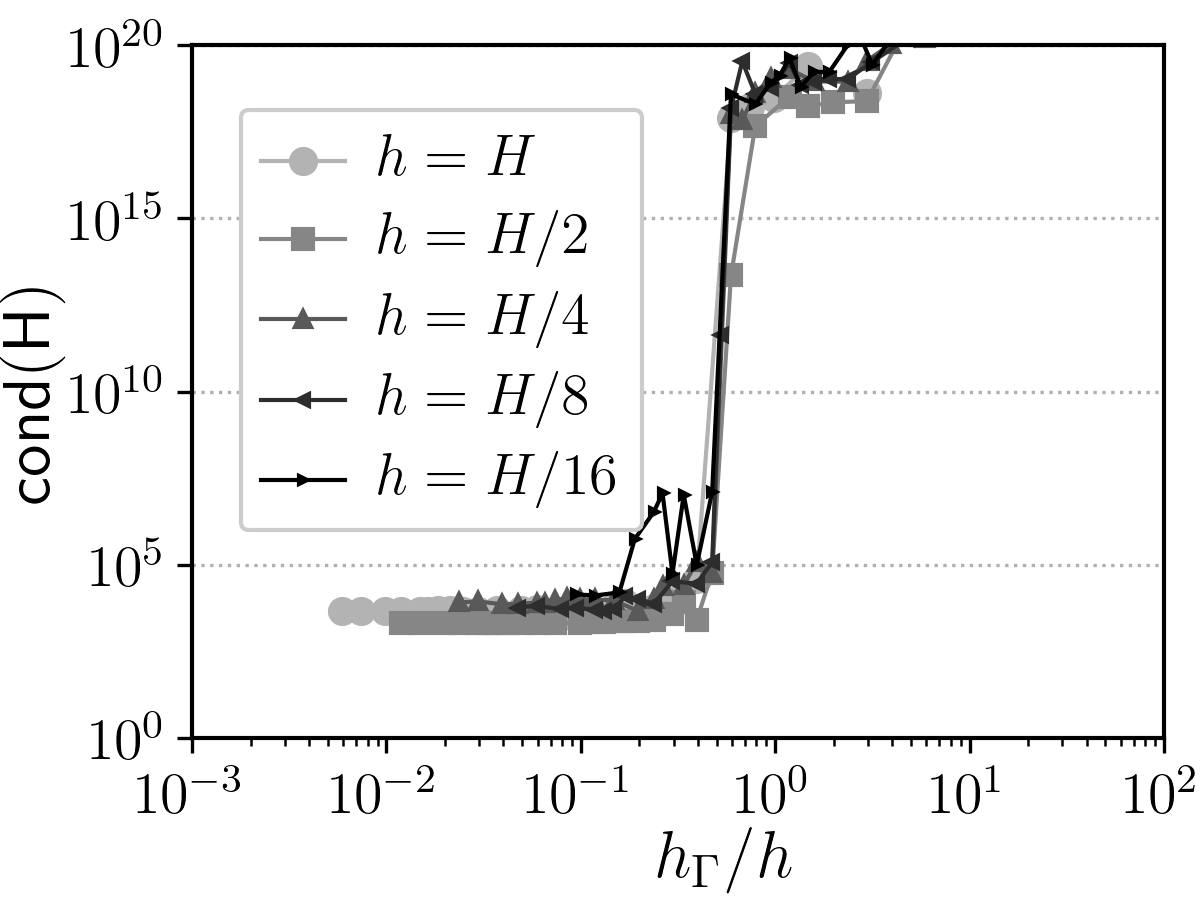}}\hfill
    \subfigure[with regularization \label{fig:cond_vs_hgamma_reg}]{%
      \includegraphics[width=0.45\textwidth]{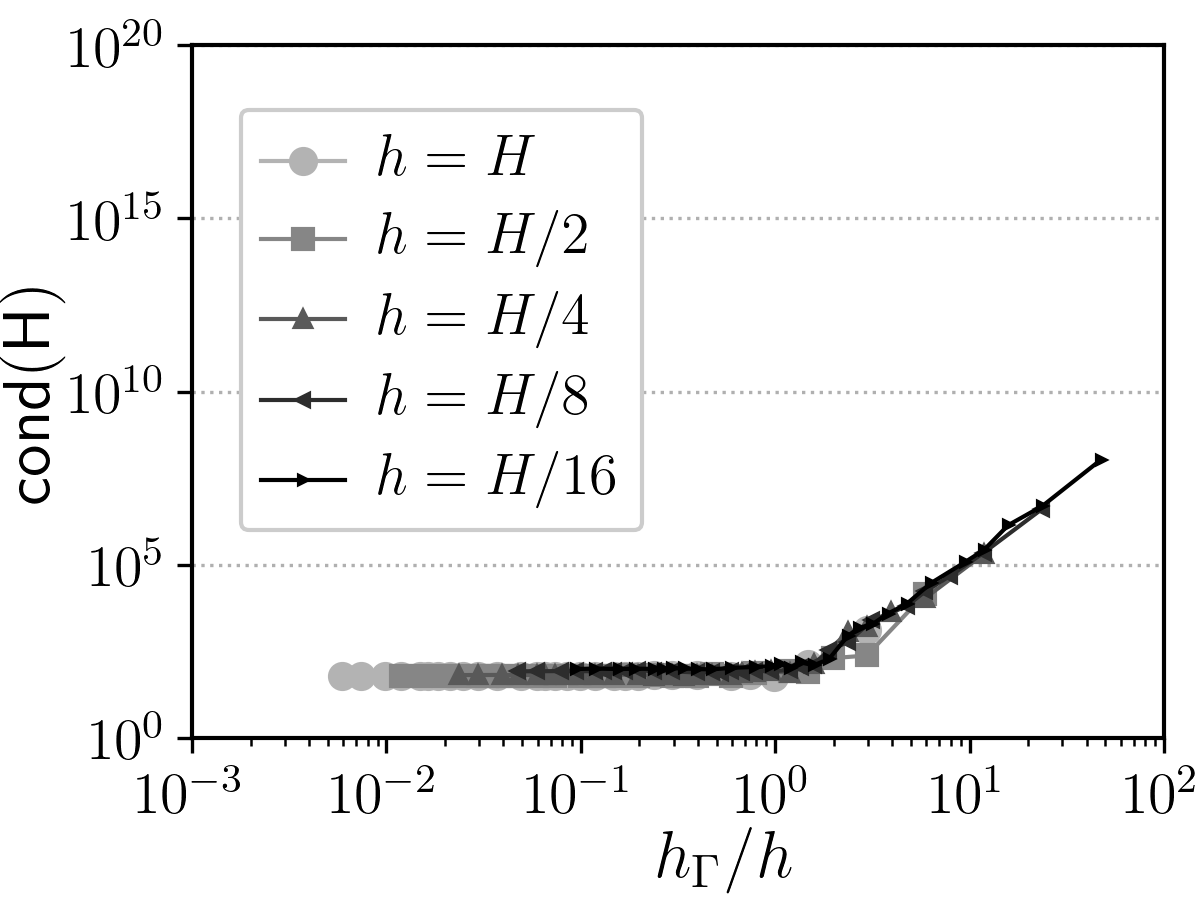}}
    \caption[]{Condition number of the reduced Hessian \eqref{eq:Hz} versus the ratio between the spacing on $\Gamma$ to the mesh spacing, for the Poisson equation. \label{fig:cond_vs_hgamma_disk}}
  \end{center}
\end{figure}

Next, we would like to study the effect of the Hausdorff distance $\dH{\Gamma}{\tGamma}$ on the conditioning of $\Hz$ like we did for the model problem in Section~\ref{sec:cond}.  Unfortunately, we cannot vary the Hausdorff distance independently from $h$ because of how $\tOmega$ is constructed; see the discussion in Section~\ref{sec:tOmega}.  We can, however, mimic the study from Section~\ref{sec:cond} by varying the degree of the polynomial basis for fixed $h$ and $\dH{\Gamma}{\tGamma}$.  That is, whereas $h$ was the control spacing in the earlier study, here we can use $h/p$ to define the nominal control spacing.

Figure~\ref{fig:cond_vs_hd} plots the condition number of the reduced Hessian versus the ratio $(h/p)/\dH{\Gamma}{\tGamma}$ for both the unregularized and regularized Poisson problems.  We consider $p=1$, $2$, and $3$ degree finite-element basis functions and a range of mesh sizes $h$ and a (dependent) range of Hausdorff distances.  The results for the unregularized problem, plotted in Figure~\ref{fig:cond_vs_hd_noreg}, behave similar to the model-problem results shown in Figures~\ref{fig:cond_vs_spacing} and \ref{fig:ellipse_cond_vs_spacing}: as the control spacing becomes small relative to the Hausdorff distance, the condition number ``blows up.''  The regularized problem --- see Figure~\ref{fig:cond_vs_hd_reg} --- displays qualitatively the same pattern, but the condition number is two orders of magnitude smaller for the same value of $(h/p)/\dH{\Gamma}{\tGamma}$.  For both the unregularized and regularized cases, we observe no significant dependence on $h$ itself.  Overall, the results in Figure~\ref{fig:cond_vs_hd} confirm the conclusions drawn earlier.

\begin{figure}[t]
  \begin{center}
    \subfigure[no regularization \label{fig:cond_vs_hd_noreg}]{%
      \includegraphics[width=0.45\textwidth]{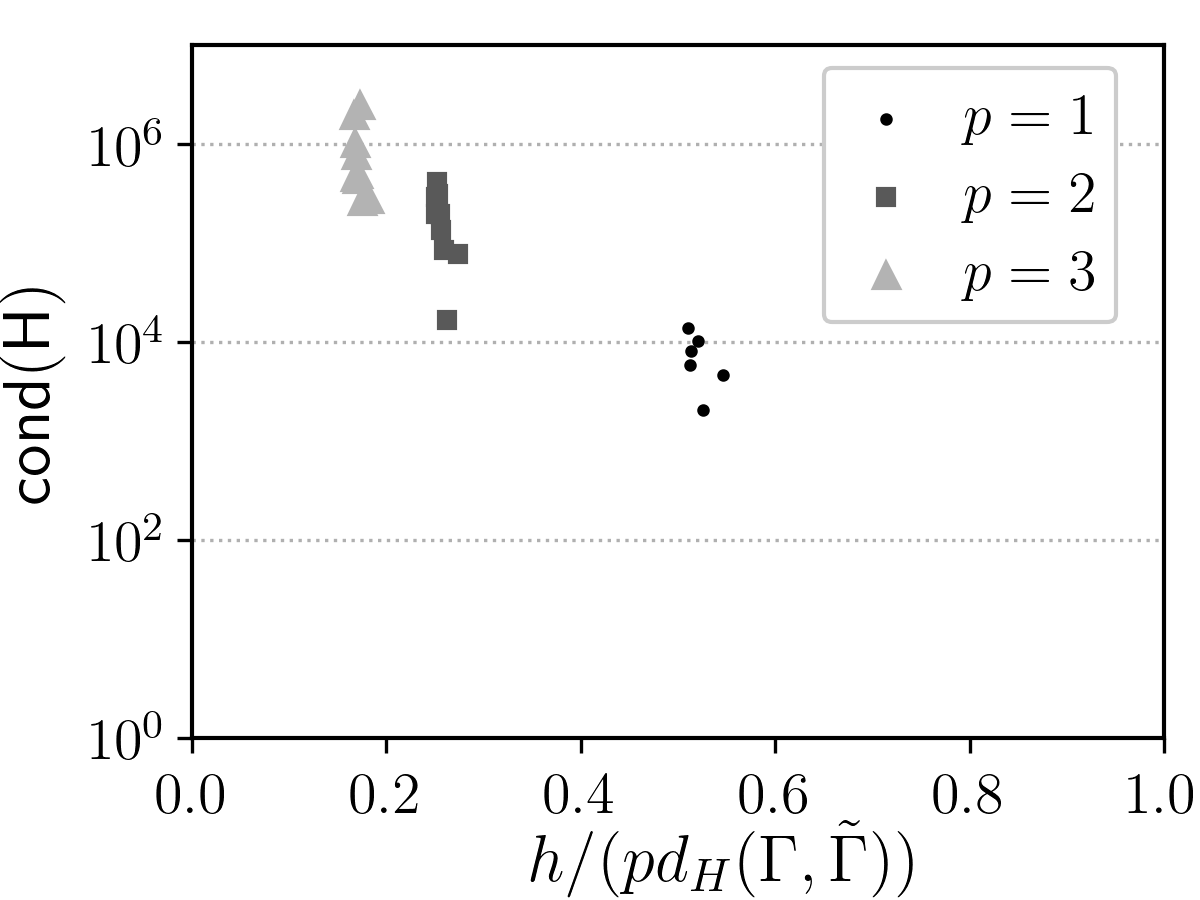}}\hfill
    \subfigure[with regularization \label{fig:cond_vs_hd_reg}]{%
      \includegraphics[width=0.45\textwidth]{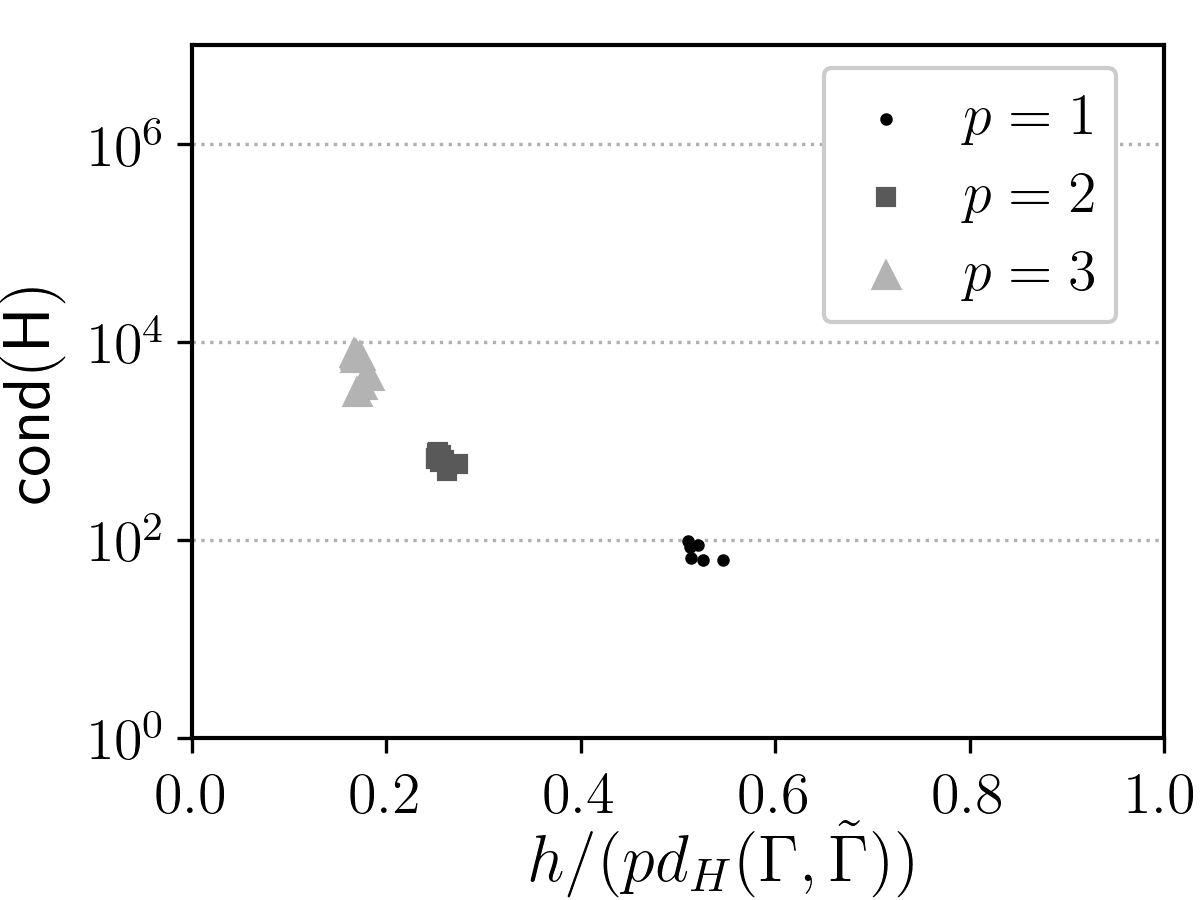}}
    \caption[]{Condition number of the reduced Hessian \eqref{eq:Hz} versus the
      ratio between the degree-of-freedom spacing, $h/p$, and the Hausdorff
      distance $\dH(\Gamma,\tGamma)$, for the Poisson equation.
      \label{fig:cond_vs_hd}}
  \end{center}
\end{figure}

\subsubsection{Linear advection equation on the unit-disk domain}\label{sec:cond_advec}

Here we investigate the conditioning of the reduced Hessian of the inverse IBM in the context of the linear advection equation on the unit circle domain.  We can only consider the regularized variant of the inverse IBM, because, unlike the Poisson problem, the unregularized advection problem is singular to working precision.

Figure~\ref{fig:cond_vs_hgamma_advec} plots the condition number of $\Hz$ versus the size ratio $h_\Gamma/h$ for the $p=1$ discretization.  We see that the regularized advection problem behaves more like the unregularized diffusion problem.  In particular, there is a threshold, $h_\Gamma/h \approx 1$ here, above which the reduced Hessian becomes singular.  Consequently, it is critical to use a sufficiently fine discretization of $\Gamma$ when the inverse IBM is applied to hyperbolic PDEs.

Figure~\ref{fig:cond_vs_hd_advec} shows how the condition number of the reduced Hessian responds to changes in the ratio $(h/p)/\dH{\Gamma}{\tGamma}$.  As with the diffusion problem, the condition number grows unbounded as the control spacing becomes small relative to the Hausdorff distance.  Furthermore, the magnitude of the condition number is larger for linear advection for the same value of $(h/p)/\dH{\Gamma}{\tGamma}$.  This suggests that the inverse IBM, with the proposed regularization, may be limited to modest values of $p$ for hyperbolic problems.

\begin{figure}[t]
  \begin{center}
    \subfigure[$\Hz$ versus $h_{\Gamma}/h$ \label{fig:cond_vs_hgamma_advec}]{%
      \includegraphics[width=0.45\textwidth]{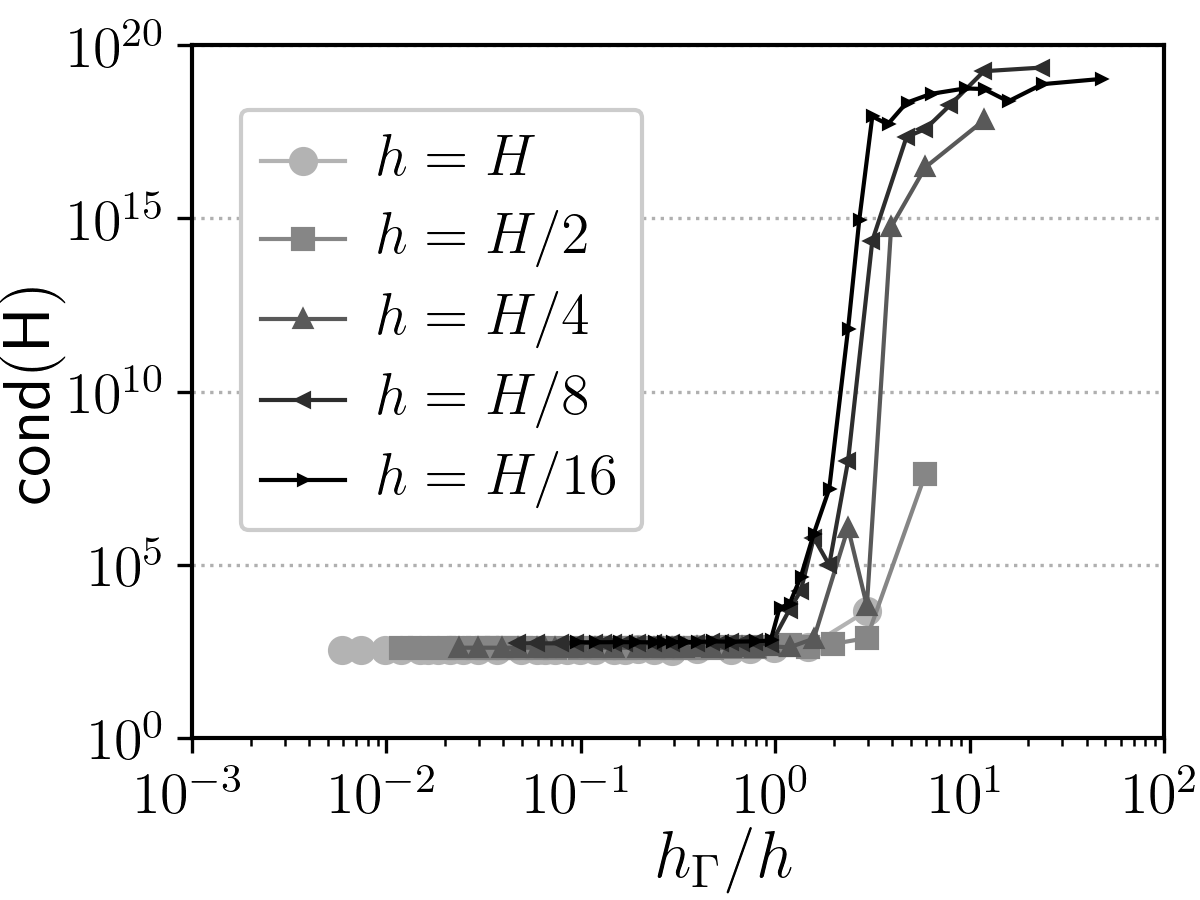}}\hfill
    \subfigure[$\Hz$ versus $h/(p \dH{\Gamma}{\tGamma})$ \label{fig:cond_vs_hd_advec}]{%
      \includegraphics[width=0.45\textwidth]{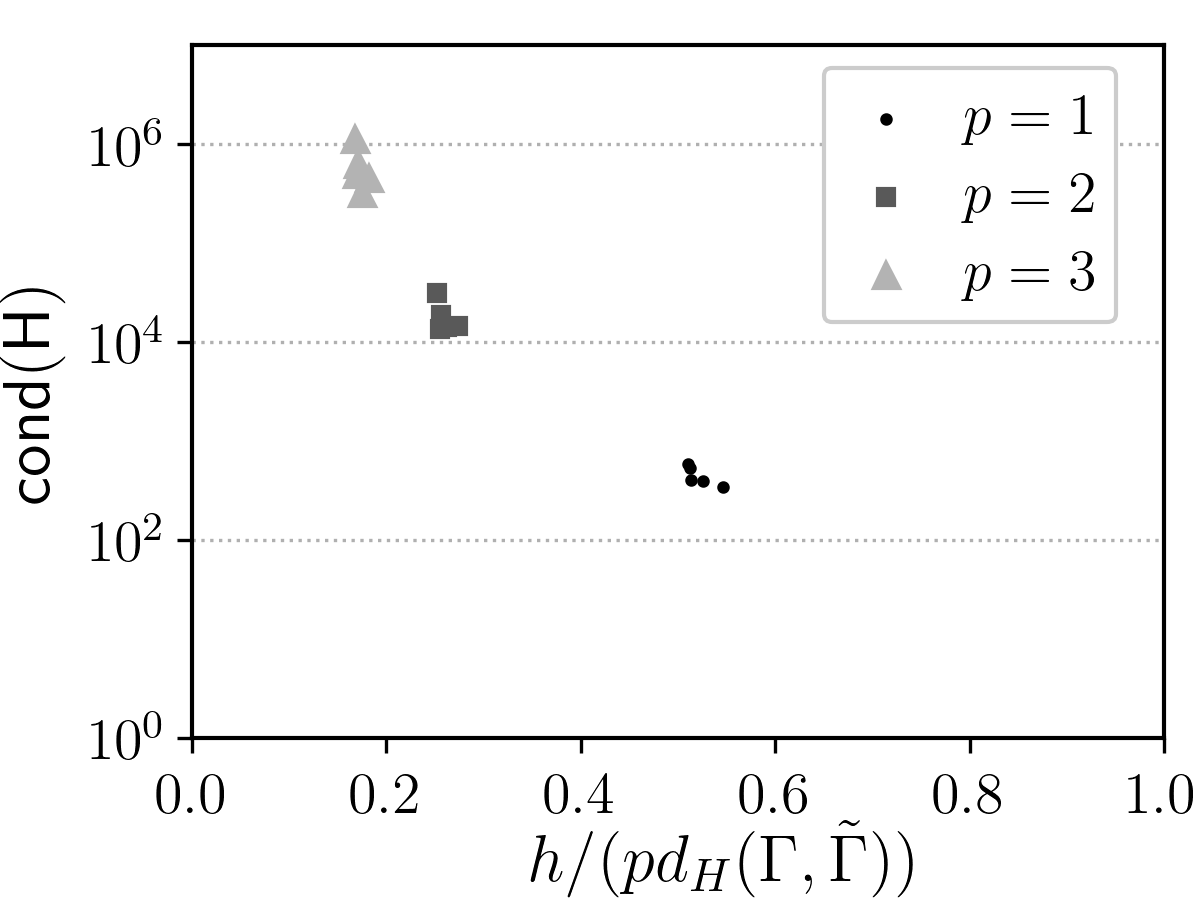}}
    \caption[]{Trends in the condition number of the reduced Hessian \eqref{eq:Hz} for the linear advection equation: dependence on $h_{\Gamma}/h$ (left) and dependence on $h/(p \dH{\Gamma}{\tGamma})$ (right).   \RevOneAdd{These results are based on the regularized inverse IBM formulation, since the Hessian is singular without regularization.} \label{fig:cond_advec}}
  \end{center}
\end{figure}

\subsubsection{Conditioning on a star-shaped (nonconvex) domain}

The results of the previous sections were obtained for a relatively benign domain; in particular, the interior of the unit-disk is convex and the boundary is smooth.  Therefore, to see if the conclusions apply to more general domains, we repeat the numerical experiments on a more complicated, star-shaped domain.  See Figure~\ref{fig:mesh_star} for the coarsest two meshes used for the star-shaped domain.

\begin{figure}[tbp]
  \centering
  \subfigure[coarsest mesh, $h=H$ \label{fig:star_mesh1}]{%
    \includegraphics[width=0.45\textwidth]{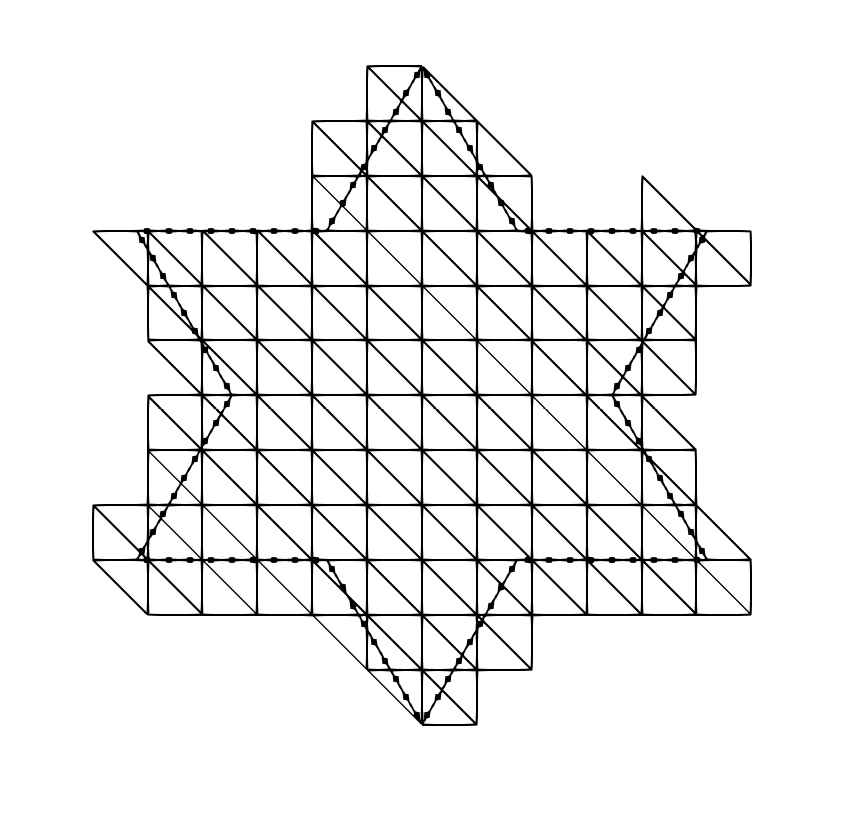}}\hfill
  \subfigure[next coarsest mesh, $h=H/2$ \label{fig:star_mesh2}]{%
    \includegraphics[width=0.45\textwidth]{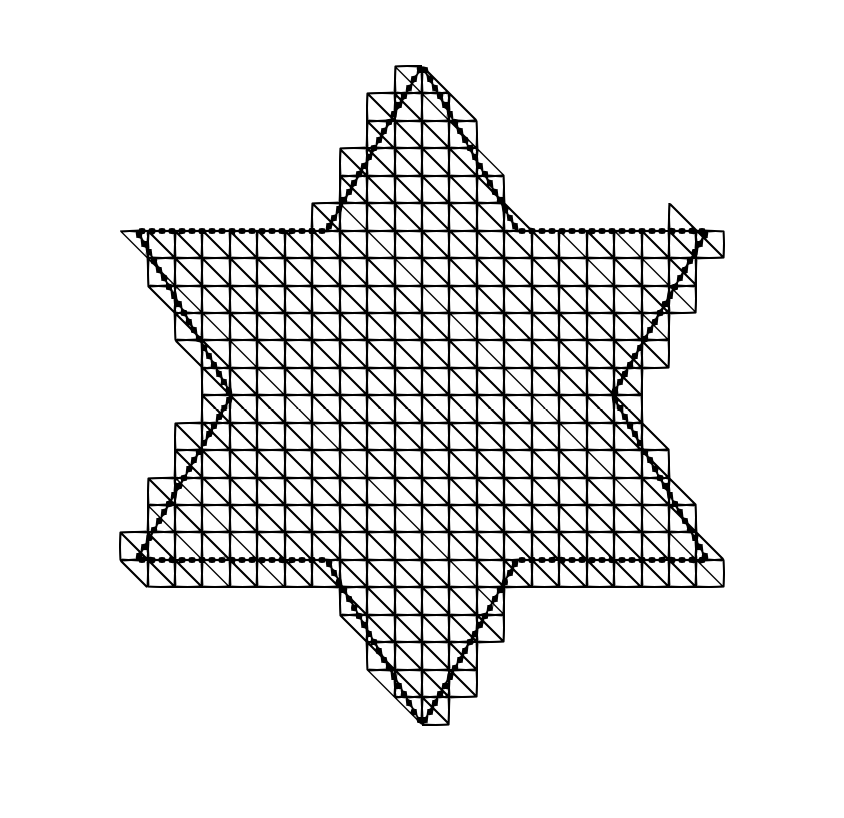}}
  \caption[]{The coarsest two meshes for the star-shape domain used for the numerical experiments.  The
    small black dots on circle denote the quadrature locations where the
    boundary-condition mismatch term in the objective is
    evaluated.  \label{fig:mesh_star}}
\end{figure}

Unlike for the unit-disk domain, the unregularized inverse IBM is singular for both the advection and diffusion problems. \RevOneAdd{We are not certain why the star-shaped domain causes issues for the unregularized formulation applied to the diffusion problem; however, we have solved problems on square domains (geometric singularities) without difficulty, so we believe the non-convex boundary is to blame.  Regardless of the source of the problem, the regularized formulation successfully resolves it, and the corresponding results are plotted in this section.}

We begin with the Poisson equation on the star-shaped domain.  Analogous to Figures~\ref{fig:cond_vs_hgamma_disk} and~\ref{fig:cond_vs_hd_reg}, Figures~\ref{fig:cond_vs_hgamma_diffn_star} and~\ref{fig:cond_vs_hd_diffn_star} plot the dependence of condition number of the reduced Hessian on the size ratio $h_{\Gamma}/h$ and the ratio $h/(p \dH{\Gamma}{\tGamma})$, respectively. We can see from the figures that the condition number behaves similarly on the star-shaped and unit-disk domains, both in terms of trend and value. 

Results for linear-advection are also similar between the unit-disk and star-shaped domains; see Figure~\ref{fig:cond_advec_star}.  One noticeable difference from the results in Figure~\ref{fig:cond_advec}, is that the condition number on the mesh corresponding to $h=H/8$ is consistently larger than on other meshes, and that oscillations in the condition number for this mesh occur around $h_\Gamma/h\approx 0.5$.  The source of this behavior for the $h=H/8$ mesh is unclear.

\begin{figure}[t]
  \begin{center}
    \subfigure[$\Hz$ versus $h_{\Gamma}/h$ \label{fig:cond_vs_hgamma_diffn_star}]{%
      \includegraphics[width=0.45\textwidth]{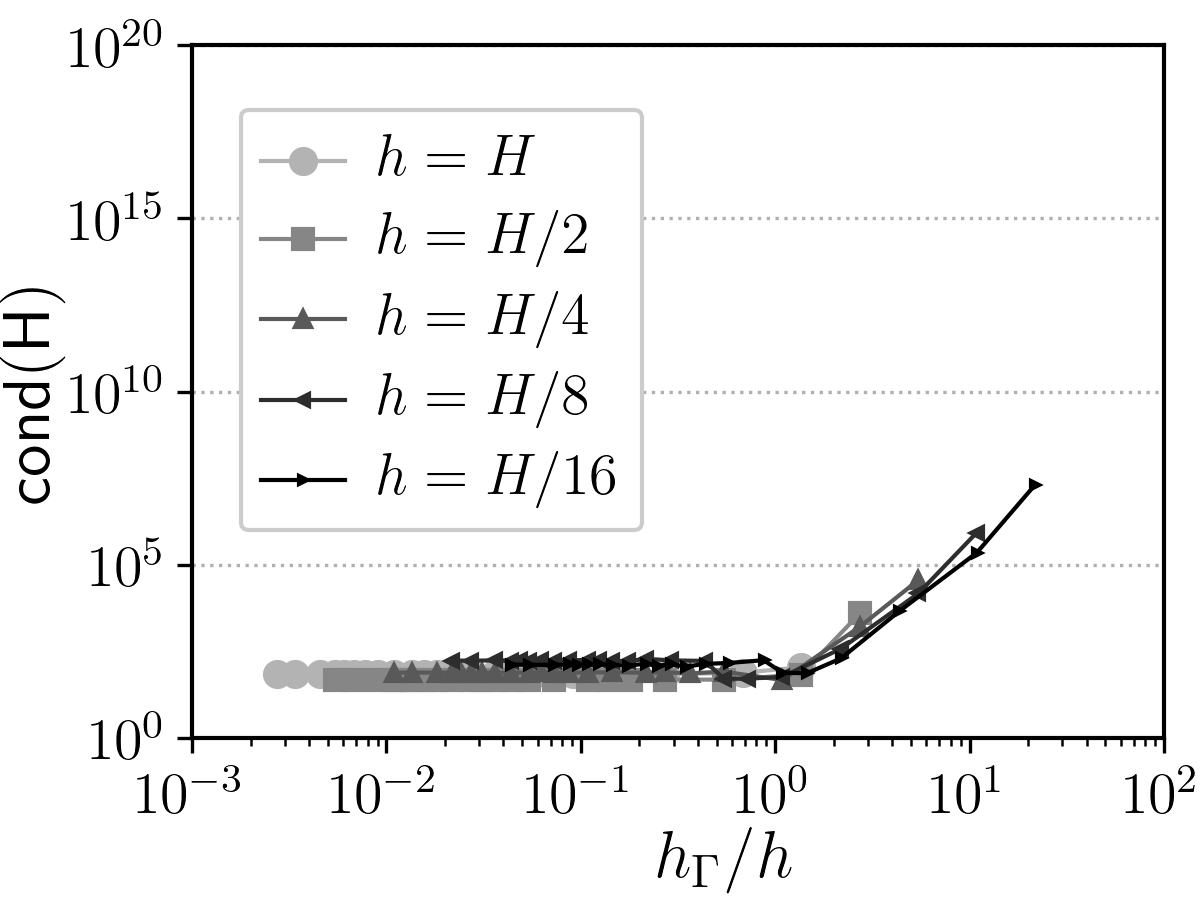}}\hfill
    \subfigure[$\Hz$ versus $h/(p \dH{\Gamma}{\tGamma})$, on the star-shaped domain \label{fig:cond_vs_hd_diffn_star}]{%
      \includegraphics[width=0.45\textwidth]{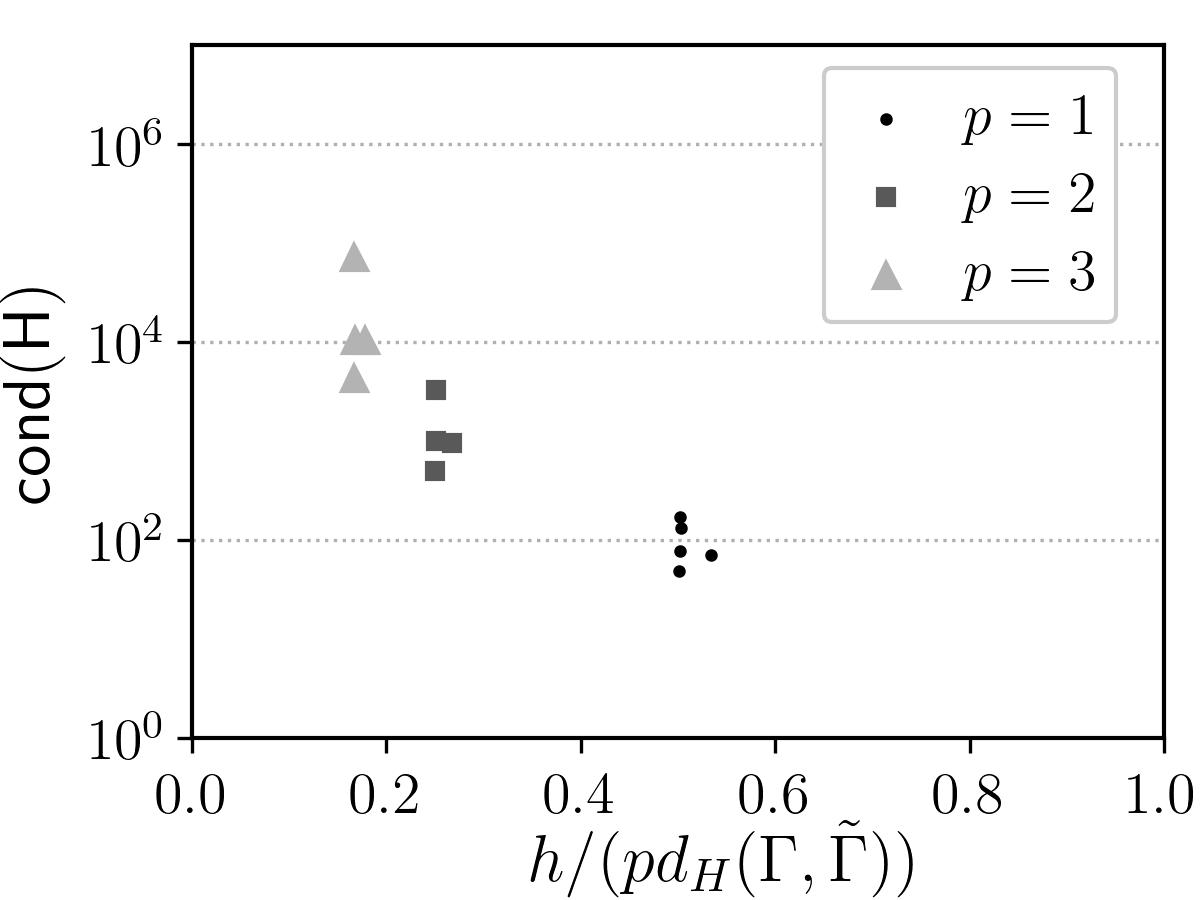}}
    \caption[]{Trends in the condition number of the reduced Hessian \eqref{eq:Hz} for the Poisson equation on \emph{the star-shaped domain}: dependence on $h_{\Gamma}/h$ (left) and dependence on $h/(p \dH{\Gamma}{\tGamma})$ (right).   \RevOneAdd{The regularized inverse IBM formulation was necessary for this non-convex problem.}  \label{fig:cond_diffn_star}}
  \end{center}
\end{figure}

\begin{figure}[t]
  \begin{center} 
    \subfigure[$\Hz$ versus $h_{\Gamma}/h$ \label{fig:cond_vs_hgamma_advec_star}]{%
      \includegraphics[width=0.45\textwidth]{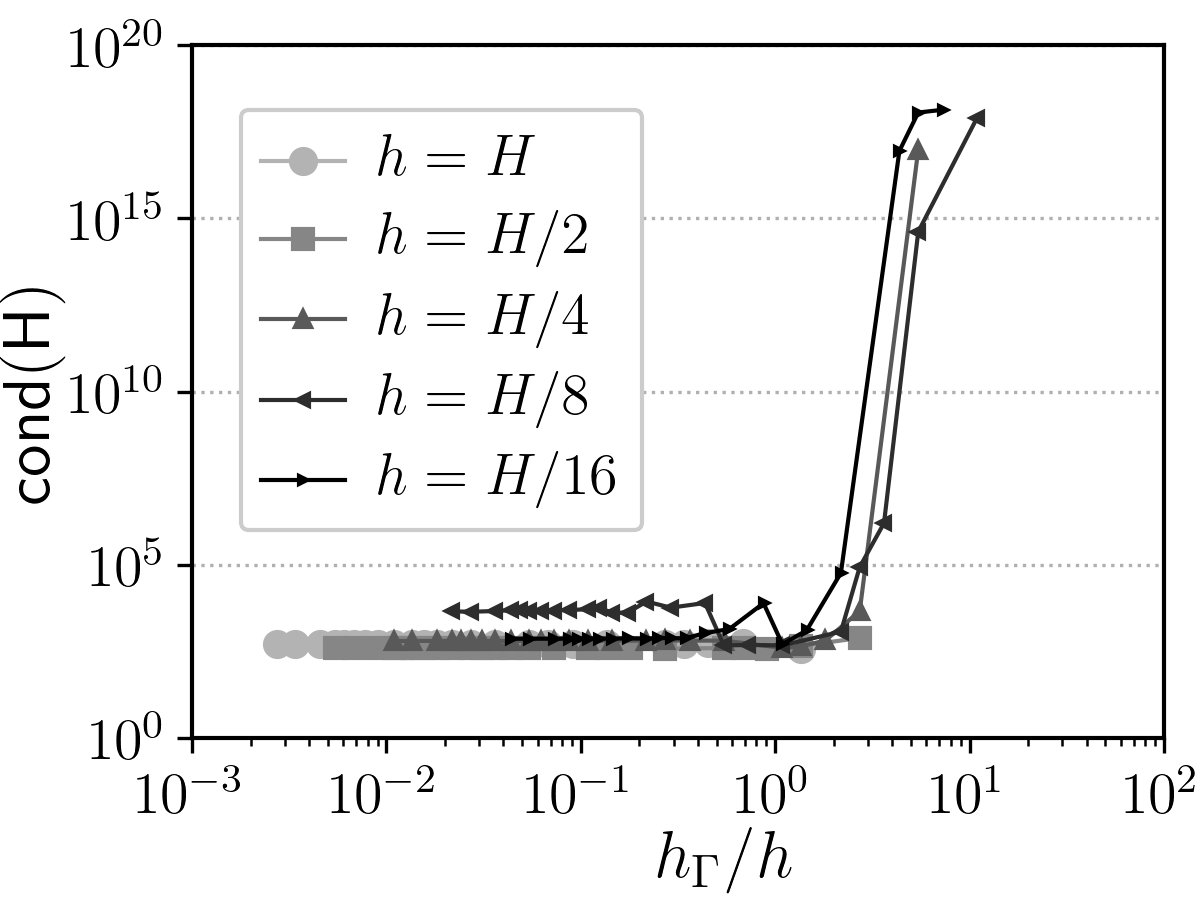}}\hfill
    \subfigure[$\Hz$ versus $h/(p \dH{\Gamma}{\tGamma})$ \label{fig:cond_vs_hd_advec_star}]{%
      \includegraphics[width=0.45\textwidth]{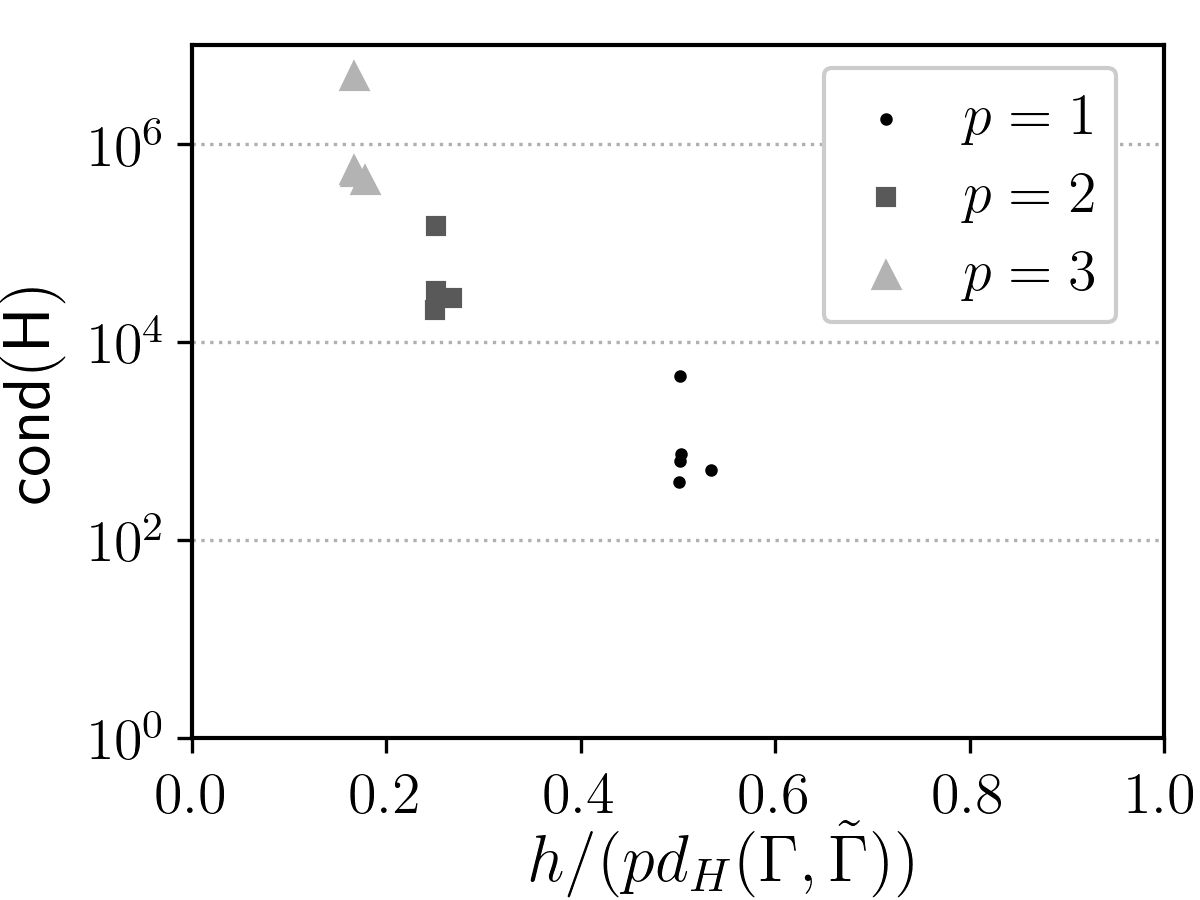}}
    \caption[]{Trends in the condition number of the reduced Hessian \eqref{eq:Hz} for the linear advection equation on \textbf{the star-shaped domain}: dependence on $h_{\Gamma}/h$ (left) and dependence on $h/(p \dH{\Gamma}{\tGamma})$ (right).  \RevOneAdd{As always with the linear advection equation, regularization was included with the inverse IBM.} \label{fig:cond_advec_star}}
  \end{center}
\end{figure}

\subsection{Convergence study}\label{sec:accuracy}

We conclude the numerical experiments with a study of solution accuracy and convergence rate. To this end, the method of manufactured solutions~\cite{Roache2001} is employed, and both smooth and nonsmooth solutions are considered.

\subsubsection{Smooth solution}\label{sec:accuracy_smooth_soln}

First a smooth exact solution defined as
\begin{equation}\label{eq:man_sln}
u(x,y) = e^{x+y}\sin(\pi x)\sin(\pi y),
\end{equation}
is selected a priori and used to determine $f$ and $\ubc$ in \eqref{eq:adv-diff}.   The two problem domains used in the condition-number study, the unit-disk and the star-shaped domains, are again considered.

To estimate the asymptotic convergence rate, we use a sequence of five uniformly refined triangular meshes for each problem domain. To obtain the next mesh in the sequence, each element is subdivided into four.  Thus, the element sizes of the finer meshes are $H/2$, $H/4$, $H/8$ and $H/16$. The coarsest two meshes for the two domains are shown in Figure~\ref{fig:mesh_disk} and~\ref{fig:mesh_star}. We can see that the physical boundary intersects mesh elements at different locations, including vertices; in our experience, the method is robust with respect to such coincident nodes.


The solution contours for the star-shaped domain using $p=1$ and $p=4$ basis functions on the coarsest mesh are compared against the exact contours in Figure~\ref{fig:contour_star_p1} and \ref{fig:contour_star_p4}, respectively. 
We can see that the discrete solution agrees well with the exact solution. Furthermore, as expected for a smooth solution, the higher-order approximation produces qualitatively better results on the same mesh.
%

\begin{figure}[tbp]
  \begin{center}
    \subfigure[$p=1$ \label{fig:contour_star_p1}]{%
      \includegraphics[width=0.45\textwidth]{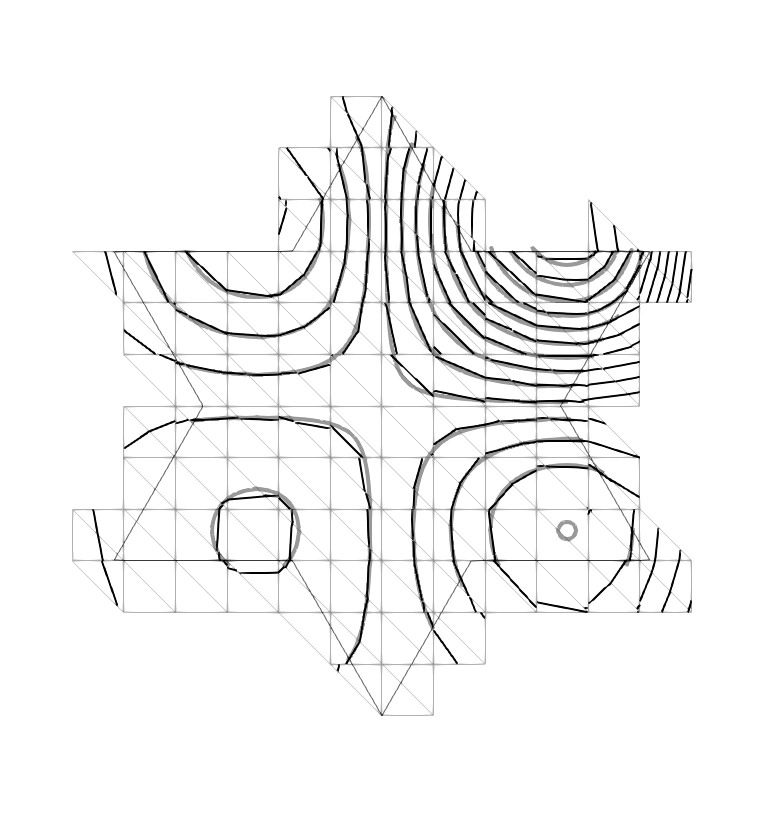}}\hfill
    \subfigure[$p=4$ \label{fig:contour_star_p4}]{%
      \includegraphics[width=0.45\textwidth]{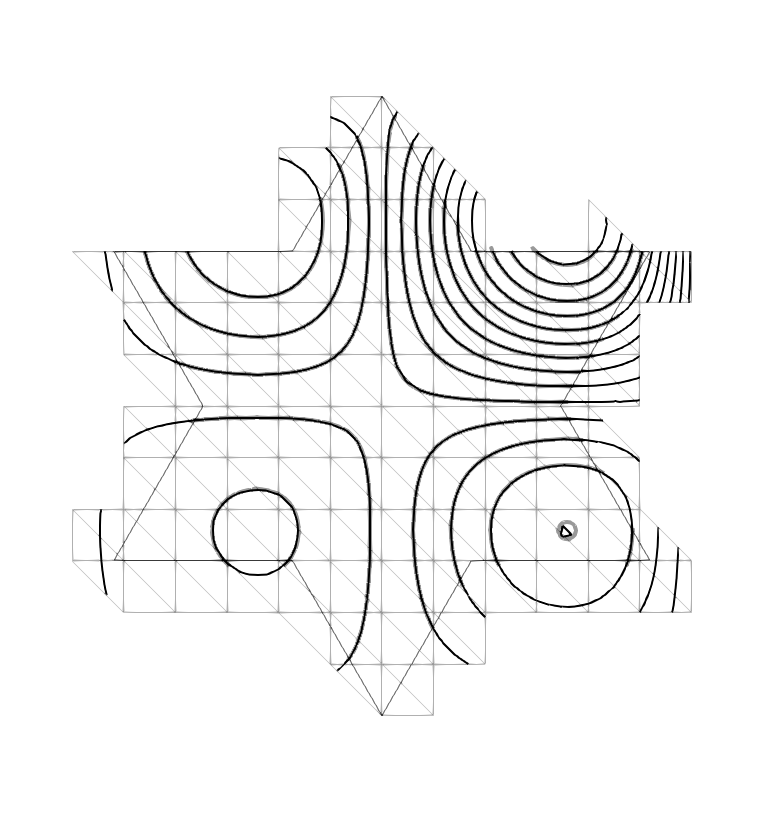}}
    \caption[]{Star domain: solution contours of the advection-diffusion problem on the coarsest mesh. Exact solution: thick gray line, discrete solution: black line.
      \label{fig:contour_star}}
  \end{center}
\end{figure}

For a quantitative assessment of accuracy, it is common to use the $L^2$ error to measure the accuracy of a finite-element solution $u_h$; however, evaluating the $L^2$ norm exactly on $\Omega$ is not straightforward, because the standard quadrature rules do not apply on the elements cut by the boundary.  The approach adopted in this paper is to set the solution error to zero at all quadrature points that lie in $\tOmega \setminus \Omega$.

Figure~\ref{fig:soln_accuracy} and~\ref{fig:soln_accuracy_star} plot the solution error versus element size $h$ for the specific advection, diffusion, and advection-diffusion problems defined earlier.  For all problems, the inverse IBM achieves the optimal convergence rate of $p+1$.

\begin{figure}[tbp]
  \begin{center}
  \rotatebox{90}{\rule{17ex}{0ex}$L^2$ \textsf{Error}}
    \subfigure[advection \label{fig:convection problem}]{%
      \includegraphics[width=0.3\textwidth]{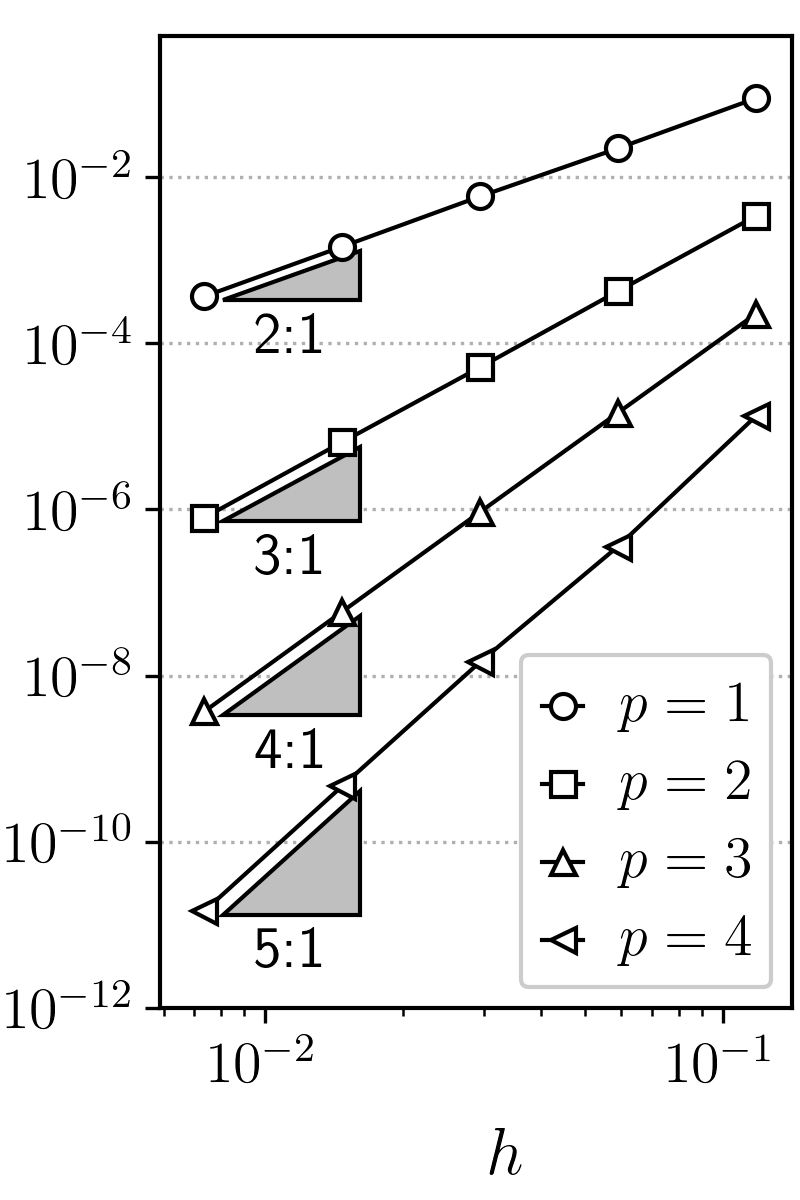}}
    \subfigure[diffusion \label{fig:diffusion problem}]{%
      \includegraphics[width=0.3\textwidth]{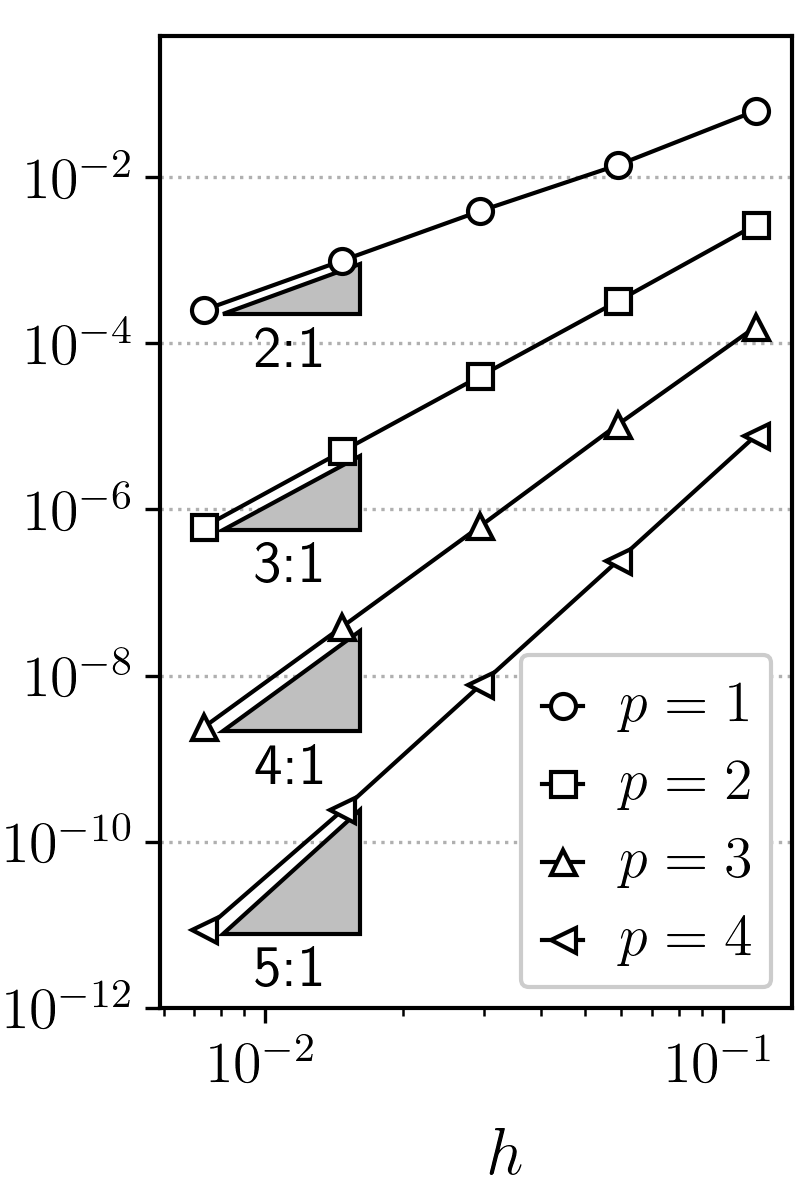}}
    \subfigure[advection-diffusion \label{fig:convection-diffusion problem}]{%
      \includegraphics[width=0.3\textwidth]{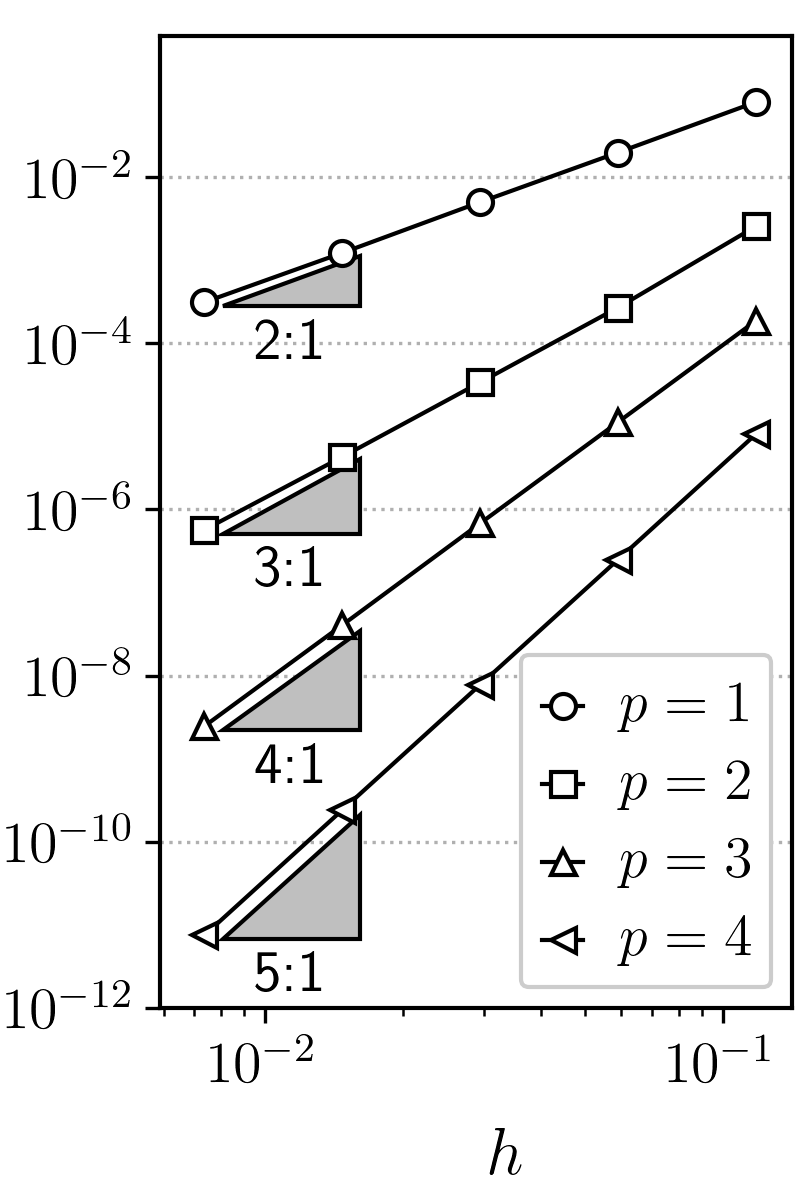}}
    \caption[]{$L^2$ solution error versus element size $h$ on the unit disk.\label{fig:soln_accuracy}}
  \end{center}
\end{figure}

\begin{figure}[tbp]
  \begin{center}
    \rotatebox{90}{\rule{17ex}{0ex}$L^2$ \textsf{Error}}
    \subfigure[advection \label{fig:convection_star}]{%
      \includegraphics[width=0.3\textwidth]{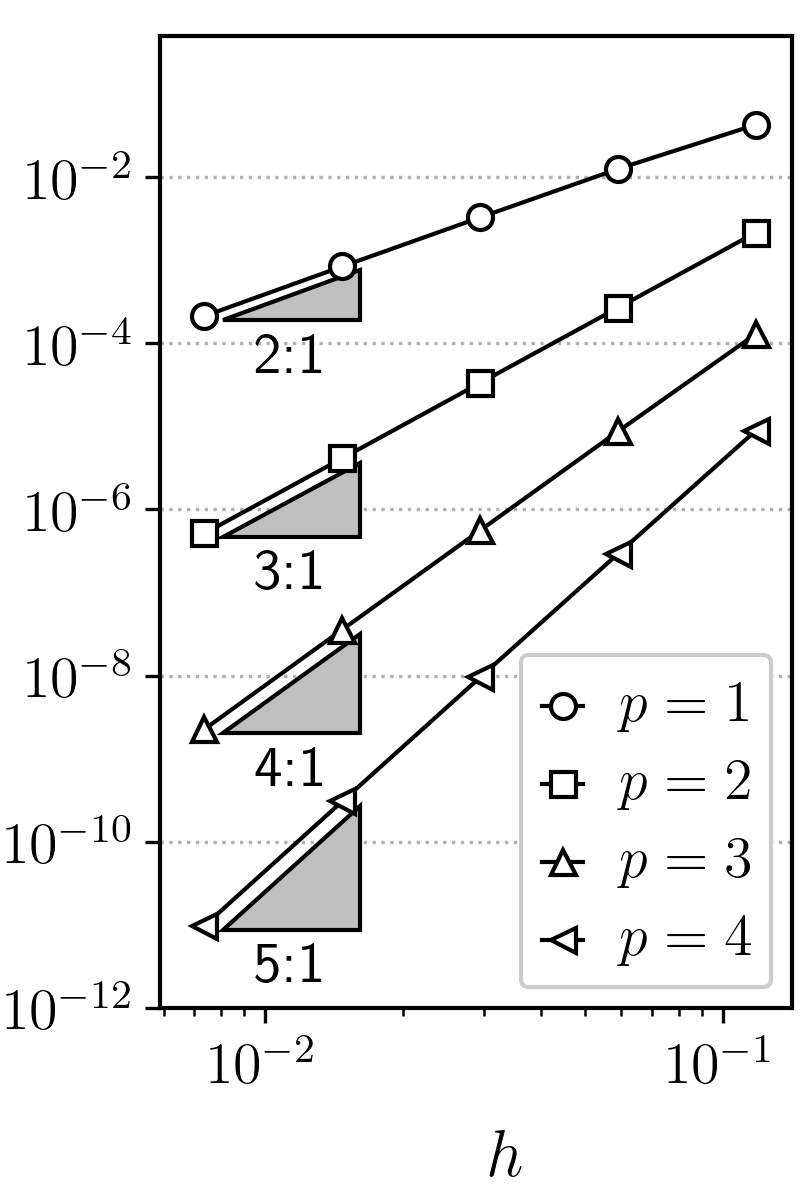}}
    \subfigure[diffusion \label{fig:diffusion_star}]{%
      \includegraphics[width=0.3\textwidth]{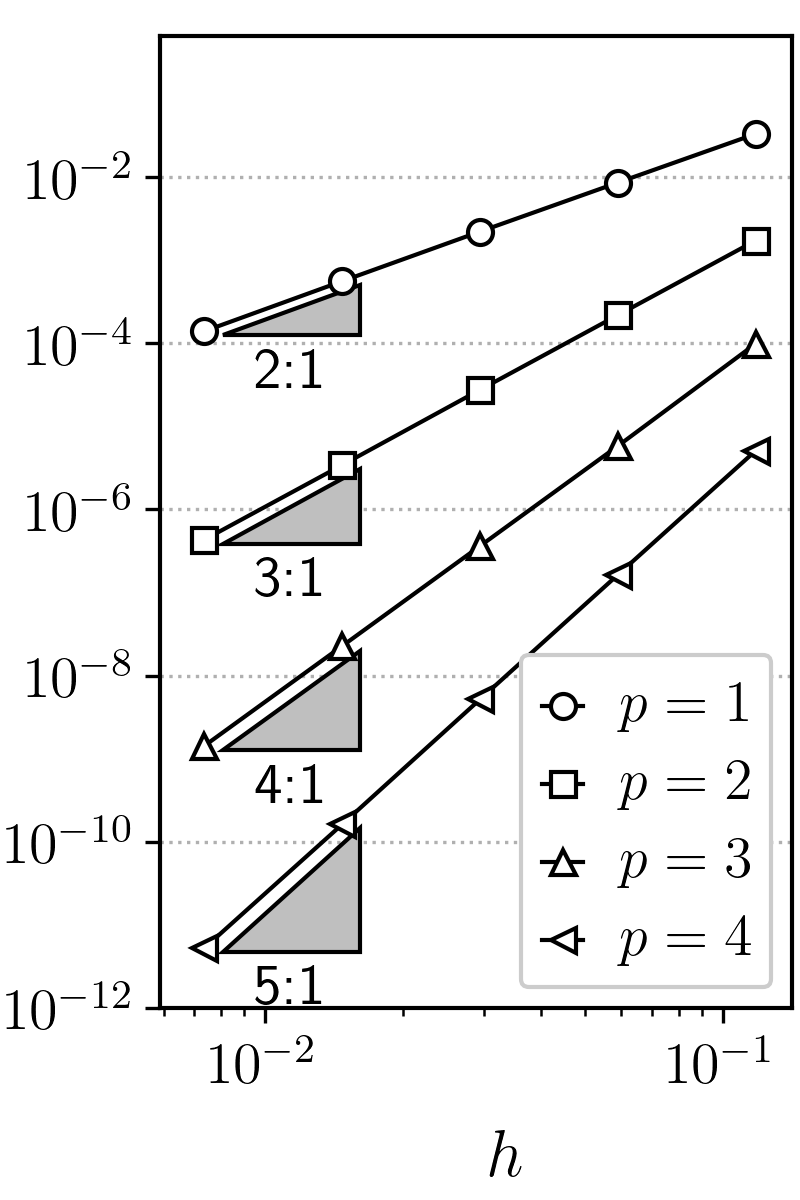}}
    \subfigure[advection-diffusion \label{fig:convection-diffusion_star}]{%
      \includegraphics[width=0.3\textwidth]{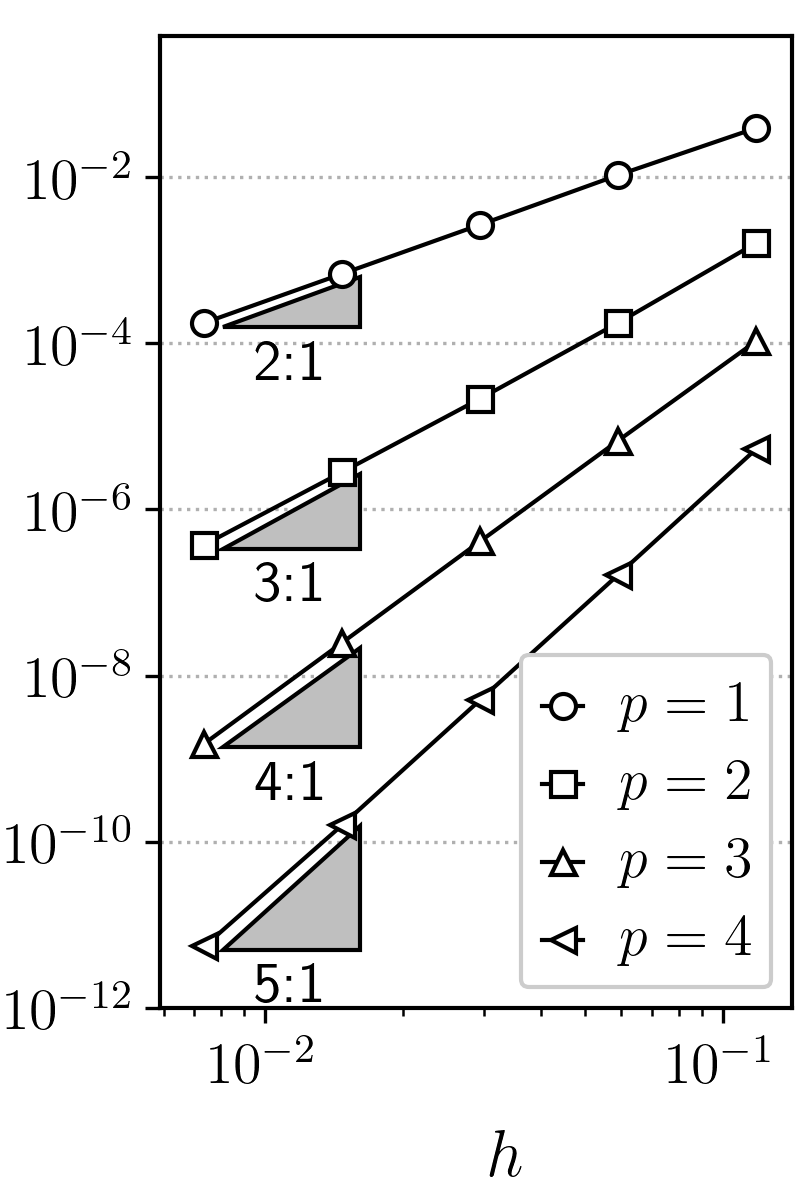}}
    \caption[]{$L^2$ solution error versus element size $h$  on the star-shaped domain.\label{fig:soln_accuracy_star}}
  \end{center}
\end{figure}

\subsubsection{Nonsmooth solution}

It is well known that the solution accuracy of a (polynomial basis) finite-element method is limited by the smoothness of the solution.  Therefore, since we use a standard DG finite-element method to discretize the state equation, we do not claim or expect the inverse IBM to be high-order accurate in the case of nonsmooth solutions.  To show this, we consider the following exact solution for the Laplace equation
on an L-shaped domain as shown in Figure~\ref{fig:lshape}:
\begin{equation}\label{eq:nonsmooth_soln}
   u=(x^2+y^2)^{1/3}\sin\left[\text{atan}\left(\cfrac{2x}{3y}\right) + \cfrac{\pi}{3}\right].
\end{equation}
This exact solution results in a zero source, $f=0$, in the Poisson equation and it has a singularity at the concave corner.

The $L^2$ errors for degree $p=1$, 2, 3, and 4 basis functions are plotted in Figure~\ref{fig:lshape_error}.  The errors are obtained on the same sequence of background meshes as in Section~\ref{sec:accuracy_smooth_soln}.  For all degrees of approximation, we obtain an asymptotic error rate of approximately $1.5$ rather than the optimal rate of $p+1$.  Nevertheless, higher degree basis functions do give a smaller error on the same mesh, which is probably due to the high resolution of the discretization in the smooth region.

\begin{figure}[tbp]
	\begin{center}
		\subfigure[The contours of the nonsmooth solution~\eqref{eq:nonsmooth_soln}. \label{fig:lshape}]{%
			\includegraphics[width=0.4\textwidth]{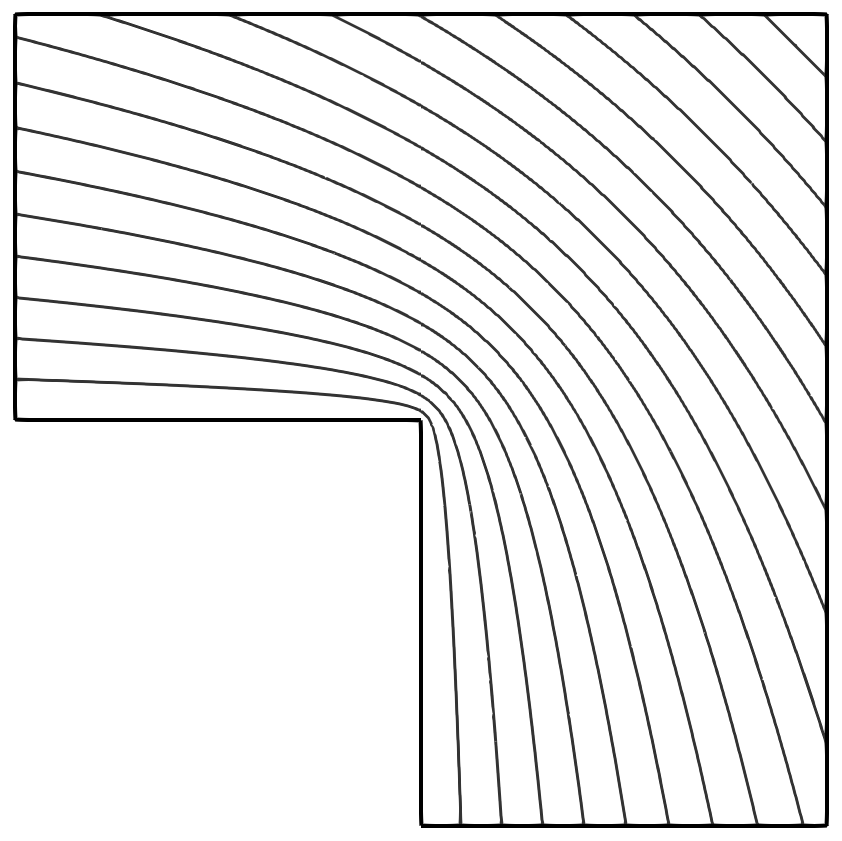}}  
		\subfigure[$L^2$ solution error versus element size h.\label{fig:lshape_error}]{%
			\includegraphics[width=0.45\textwidth]{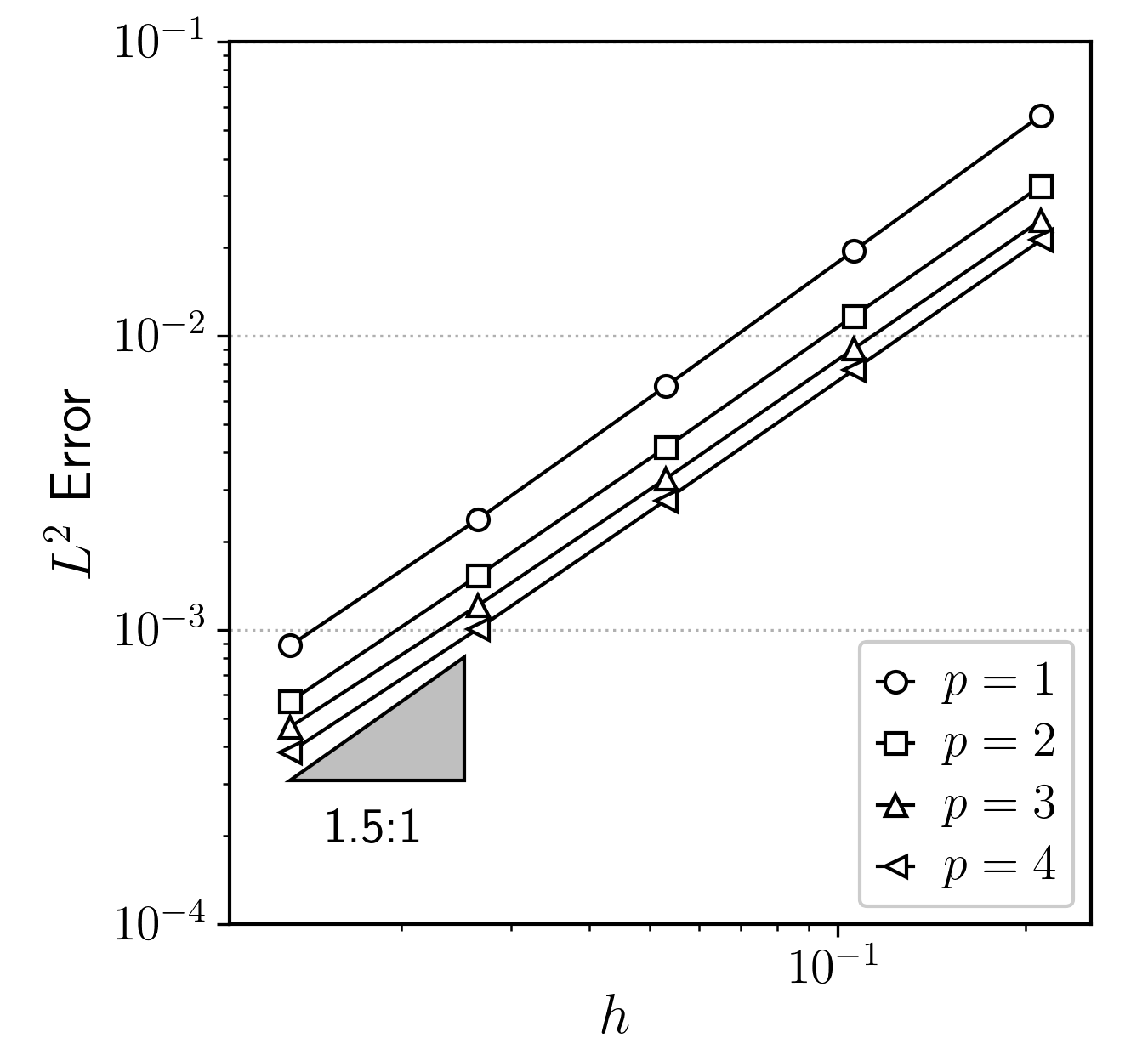}}
		\caption[]{Convergence study for the Poisson equation on the L-shaped domain.\label{fig:conv_lshape}}
	\end{center}
\end{figure}

\section{Summary and Discussion}\label{sec:conclude}

First, a brief summary.  The proposed inverse IBM introduces a control on the boundary of an expanded domain $\tOmega$ that encompasses the target domain $\Omega$.  In the basic formulation, the value of the control is determined by minimizing the mismatch in the boundary condition on the true boundary $\Gamma$ subject to the state satisfying the desired PDE on the expanded domain $\tOmega$.  This basic formulation is ill-posed for the Laplace and linear advection equations: for the former, the solution does not depend continuously on the data, and for the latter the solution is not uniquely determined by the data.  The ill-posedness is addressed by ensuring the expanded domain tends to the target domain and by including a control-state penalty term in the objective. Applied to a discontinuous Galerkin finite-element discretization of the advection, advection-diffusion, and diffusion equations, the inverse IBM formulation achieves optimal solution convergence rates (for smooth solutions), and the reduced Hessian remains well-conditioned.

We believe the inverse IBM has a number of attractive features.  The approach is agnostic to the underlying discretization, so it can be applied to finite difference, finite volume, and finite-element discretizations.  Furthermore, the method is compatible with high-order discretizations as our results demonstrate.  Finally, once the discretization is chosen, the implementation is straightforward in the sense that there are no ``corner cases'' that require special treatment; all that is required is an interpolation/projection operator from $\tOmega$ to the boundary $\Gamma$.  

There is no ``free lunch,'' and the attractive features of the inverse IBM come at a price.  Specifically, the method trades computational cost for ease of implementation and accuracy.  If there are $n$ state variables and $m$ control variables, the inverse IBM has $2n +m$ variables.  This should be contrasted with other approaches that have only $n$ state variables.  In addition, the inverse IBM is a PDE-constrained optimization problem, which requires the solution of a linear saddle-point problem for linear PDEs and a nonlinear system in general.  Nevertheless, we believe the increased computational cost is well worth the reduced human time required in mesh generation around complex configurations.

The proposed inverse IBM is promising, but there remain many challenges that must be addressed before it can be used on practical problems.  Foremost among these challenges is the robust and efficient solution of the nonlinear PDE-constrained optimization problem associated with the inverse IBM.  A related issue is the need for specialized preconditioners for the saddle-point systems that arise in the Newton iterations.  Finally, there is the question of how the method can be extended to unsteady simulations.

\appendix

\section{Discretization of the advection-diffusion equation}\label{app:disc}
We begin with introducing some additional notation.
Let $\tOmega_h$ be a shape-regular subdivision of $\tOmega$ into disjoint triangular elements $K \in \tOmega_h$; the definition of $\tOmega$ is provided in Section~\ref{sec:tOmega}.  The space of discontinuous piecewise polynomials of total degree $p$ on $\tOmega_h$ is
\begin{equation*}
  \Vhp(\tOmega_h) \equiv  \{ v_h \in L^2(\tOmega) | \left.v_h \right|_{K} \circ \sigma_{K} \in \mathbb{P}^{p}(\hat{K}), K \in \tOmega_h\},
\end{equation*}
where $\mathbb{P}^{p}(\hat{K})$ is the space of polynomials of total degree $p$ on the unit triangle $\hat{K}$, and $\sigma_{K}$ denotes a push-forward operation from $K$ to $\hat{K}$.  In this work we use nodal Lagrange bases of degree one to four.

To define the discrete space for the control, we introduce $\tGamma_h$ as the subdivision of $\tGamma$ corresponding to the boundary faces of the elements in $\tOmega_h$, \ie $\tGamma_h = \{ \bar{K} \cap \tGamma \neq \emptyset, K \in \tOmega_h \}$.  Then, the discrete control space is the set of discontinuous piecewise polynomials of total degree $p$ on $\tGamma_h$:
\begin{equation*}
  \Vhp(\tGamma_h) \equiv \{ d_h \in L^2(\tGamma) | \left.d_h \right|_{e} \circ \sigma_e \in \mathbb{P}^{p}(\hat{e}), e \in \tGamma_h \}.
\end{equation*}
Here, $\mathbb{P}^{p}(e)$ is the space of polynomials of degree $p$ on the reference interval $\hat{e} = [-1,1]$, and $\sigma_e$ denotes push-forward operation from edge $e$ to $\hat{e}$.

Next, we define the standard jump and average operators on the set of interior faces, $\Gamma_I$.  For an interior face $e \in \Gamma_{I}$, the jump and average operators, applied to the scalar $u_h \in \Vhp(\tOmega_h)$ and vector $q_h \in [\Vhp(\tOmega_h)]^2$, are given by
\begin{equation*}
  \begin{aligned}
    \Mean{u_h} = (u_h^+ + u_h^-)/2,   \qquad& \Mean{q_h} = (q_h^+ + q_h^-)/2, \\
    \Jump{u_h} = u_h^+ n^+ + u_h^-n^-, \qquad & \Jump{q_h} = q_h^+\cdot n^+ + q_h^-\cdot n^-,
  \end{aligned}
\end{equation*}
where $n^+$ and $n^-$ are the outward unit normals with respect to $\partial K^+$ and $\partial K^-$, respectively, and $u_h^+$ and $u_h^-$ are the traces of $u_h$ along the common face from the interior of $K^+$ and $K^-$, respectively.  The traces $q_h^+$ and $q_h^-$ are defined similarly.

The trilinear weak form corresponding to the DG discretization of~\eqref{eq:adv-diff} is a mapping $\b_h : \Vhp(\tOmega_h) \times \Vhp(\tGamma_h) \times \Vhp(\tOmega_h) \rightarrow \mathbb{R}$ defined by 
\begin{equation}\label{eq:adv-diff_DG}
\begin{aligned}
\b_h(u_h, c_h, v_h) =& \int_{\tilde{\Omega}_h} \left[\nabla v_h\cdot (\lambda u_h - \mu\nabla u_h) +v_hf_h\right]\diff \Omega \\
-& \int_{\Gamma_{I}}\left[\Jump{v_h}\cdot\hat{F}(u_h^+, u_h^-)\right] \diff\Gamma 
- \int_{\tGamma}\left[v_h \hat{F}(u_h, c_h)\cdot n\right] \diff\Gamma\\
+& \int_{\Gamma_{I}}\left[\Jump{v_h}\cdot \Mean{\mu\nabla u_h}
+ \Mean{\mu\nabla v_h}\cdot \Jump{u_h} 
- \epsilon \Mean{\mu}\Jump{u_h}\cdot\Jump{v_h}\right] \diff\Gamma \\ 
+& \int_{\tGamma} \left[v_h \mu\nabla u_h\cdot n 
+ (\mu\nabla v_h)\cdot n(u_h-c_h)
-\epsilon \mu v_h (u_h - c_h)\right] \diff\Gamma,
\end{aligned}
\end{equation}
where $\hat{F}$ denotes the standard upwind flux function, and $\epsilon$ is the SIPG penalty parameter based on~\cite{Shahbazi2005explicit}. The discrete source, $f_h$, is obtained by projecting the exact source $f$, derived from the manufactured solution~\eqref{eq:man_sln}, onto $\Vhp(\tOmega_h)$.

\ack

\section*{Acknowledgements}

The authors are grateful to Fengyan Li for looking over the manuscript, answering our functional analysis questions, and helping show that the mapping $T$ was unbounded in Section~\ref{sec:ill-posed}.  The authors were supported by the National Science Foundation under Grant No. 1825991 and gratefully acknowledge this support.

\bibliographystyle{WileyNJD-AMA}
\bibliography{references}

\end{document}